\documentclass{amsart}

\usepackage{amsmath,amssymb,mathtools}
\usepackage{microtype}
\usepackage[hidelinks,hyperfootnotes=false]{hyperref}

\allowdisplaybreaks
\numberwithin{equation}{section}

\theoremstyle{plain}
\newtheorem{theorem}{Theorem}[section]
\newtheorem{proposition}[theorem]{Proposition}
\newtheorem{lemma}[theorem]{Lemma}
\newtheorem{corollary}[theorem]{Corollary}

\theoremstyle{definition}
\newtheorem{definition}[theorem]{Definition}

\theoremstyle{remark}
\newtheorem{remark}[theorem]{Remark}
\newtheorem{example}[theorem]{Example}

\newcommand{\C}{\mathbb C}
\newcommand{\D}{\mathbb D}
\newcommand{\T}{\mathbb T}
\newcommand{\PSH}{\mathrm{PSH}}
\newcommand{\cA}{\mathcal A}
\newcommand{\dist}{\operatorname{dist}}
\newcommand{\ord}{\operatorname{ord}}
\newcommand{\ddc}{dd^c}
\newcommand{\eps}{\varepsilon}

\begin{document}

\title[Green functions on certain B-regular domains]
{Continuity of pluricomplex Green functions with hypersurface poles on certain B-regular domains}

\author[H. H. Dao]{Huy Hoang Dao}
\address{Department of Mathematics, Hanoi National University of Education, Hanoi, Vietnam}
\email{daohuyhoang05121999@gmail.com}

\author[Q. D. Nguyen]{Quang Dieu Nguyen}
\address{Department of Mathematics, Hanoi National University of Education, Hanoi, Vietnam}
\email{ngquang.dieu@hnue.edu.vn}

\author[D. H. Tran]{Duc Hieu Tran}
\address{Hanoi-Amsterdam School for the Gifted, Hanoi, Vietnam}
\email{tranduchieu1709@gmail.com}

\subjclass[2020]{32U05, 32U15, 32H02}
\keywords{Pluricomplex Green function, complex hypersurface, analytic disc, Weierstrass preparation, sectorial deformation, Hartogs hull, Perron--Bremermann envelope, quasi upper bounded function, convex domain}
\date{}

\begin{abstract}
Let $\Omega\Subset\C^n$ and let $f$ define a reduced hypersurface divisor on a neighborhood of $\overline\Omega$, with $A=(f)|_\Omega$.
Our main result proves canonical bounded-envelope approximation of the desingularized Green function whenever the pair $(\Omega,A)$ admits a sectorial inward deformation (SID); if $\Omega$ is B-regular, this yields extended-real continuity of $G_{A,\Omega}$.
As an important consequence, the continuity conclusion holds for every B-regular convexifiable domain, without any smoothness assumption on the boundary trace of $A$.
The proof combines moving centers, parametric Weierstrass preparation, and a quantitative analytic-disc insertion argument whose Cartan gluing is arranged so that all target Rouch\'e margins are fixed before the large source root-distribution parameter; a bounded Kohn--Nirenberg-type model shows that the SID theorem goes beyond convexifiability.
\end{abstract}

\maketitle

\section{Introduction and preliminaries}\label{sec:intro}
\enlargethispage{3pt}

Let $\Omega\Subset\C^n$ be a bounded domain.  If
$u\in\PSH(\Omega)$, we denote by $\nu_u(z)$ its Lelong number at $z$.  If
$A$ is an effective divisor, then $\nu_A(z)$ denotes the multiplicity of $A$
at $z$; in the principal case $A=(f)$ one has
$\nu_A=\nu_{\log|f|}$.  Following L\'arusson and Sigurdsson \cite{LS99}, the
pluricomplex Green function with pole along $A$ is
\begin{equation}\label{eq:defG}
 G_{A,\Omega}(z)
 :=\sup\bigl\{u(z):u\in\PSH(\Omega),\ u\le0,\ \nu_u\ge\nu_A\bigr\},
 \qquad z\in\Omega.
\end{equation}
In this paper the divisor is principal: $A=(f)|_\Omega$, where
$f\not\equiv0$ is holomorphic on a neighborhood of $\overline\Omega$.  When
boundary points are involved, we use the same symbol $A$ for the zero set
$\{f=0\}$ in this fixed neighborhood; thus $A\cap\partial\Omega$ means
$\{f=0\}\cap\partial\Omega$.  We always assume
\[
 \{f=0\}\cap\Omega\ne\varnothing,
\]
since otherwise the pole set is empty and the continuity question considered
below is trivial.  Throughout the SID and convexifiability results, the extending divisor
$(f)$ is assumed to be reduced on a fixed open neighborhood of
$\overline\Omega$, not merely after restriction to $\Omega$.  Here reduced
means that $f$ has no repeated local irreducible factor, up to multiplication
by a nowhere-vanishing holomorphic function.  The notation
$\operatorname{Reg}(A)$ and $\operatorname{Sing}(A)$ near the boundary refers
to this extending reduced hypersurface.  In particular,
$df(p)\ne0$ at every $p\in\operatorname{Reg}(A)$, including boundary points.
Singular points of the extending hypersurface are allowed.  These standing
assumptions will not be repeated in every statement; Proposition~\ref{prop:Hartogs-fibers}
explicitly allows a positive multiplicity.

The function \eqref{eq:defG} was introduced by L\'arusson and Sigurdsson in
\cite{LS99} for poles along complex subspaces, and more generally in terms of
coherent ideal sheaves.  It extends the classical pluricomplex Green function
with one pole, as well as the case of finitely many weighted poles.  Their
point of view connects Lelong numbers, extremal plurisubharmonic functions,
and envelopes of analytic-disc functionals, and makes it possible to treat
positive-dimensional analytic pole sets.  Green functions with analytic
singularities, powers of ideals, and colliding poles were later studied from
other directions; see, for example, \cite{RS05,RT13}.

A basic regularity problem is the continuity of $G_{A,\Omega}$ away from its
pole set.  This question is already subtle when a hypersurface meets the
boundary of the domain.  In \cite[Theorem~3.9]{LS99}, L\'arusson and
Sigurdsson claimed an ``almost continuity'' statement under strong boundary
barriers.  Their proof considers bounded maximal solutions with boundary data
$\min\{-\log|f|,j\}$ and then identifies the upper regularization of their
increasing limit with the desingularized Green function.  The erratum
\cite{LSerr} points out that the comparison needed for this last
identification is not justified.  The difficulty occurs precisely at
$A\cap\partial\Omega$, where the limiting boundary datum $-\log|f|$ becomes
unbounded.  We shall use \cite[Proposition~3.2]{LS99} and
\cite[Theorem~3.3]{LS99}; these results are independent of the comparison
step withdrawn in the erratum.

Dieu \cite{Dieu05} later proved continuity in several geometric situations.
The method there is based on holomorphic hulls of Hartogs sets and on a radial
inward deformation near the boundary part of the divisor.  One of the main
ideas of \cite{Dieu05} is to encode the boundary modulus $|f|$ by a compact
Hartogs set in $U\times\C$ and to read a lower semicontinuous fiber function
from its holomorphic hull.  The comparison of this fiber function with the
desingularized Green function turns the continuity problem into a hull
problem.  We return to this Hartogs-hull viewpoint in Section~\ref{sec:hartogs},
where the bounded truncations admit a direct interpretation in terms of
vertical hull radii.

In the principal result of \cite{Dieu05}, the divisor is a finite union of
complex hyperplanes through a common point and a relative neighborhood of its
boundary trace can be pushed inside the domain by radial dilations.  The
argument uses the homogeneity of a defining polynomial in an essential way.
It therefore gives strong positive results for special pole geometry, but it
does not directly extend to a general hypersurface whose order of contact may
vary from one boundary point to another.  Proposition~3.1 and the remarks
following it in \cite{Dieu05} also show that hyperconvexity alone cannot replace
B-regularity in the continuity problem.  Remark~(a) gives a bounded hyperconvex
non-B-regular Hartogs domain for which the relevant Green function still has
good continuity behavior, whereas Remark~(b), starting from a bounded
discontinuous subharmonic function on the unit disc, gives a bounded
hyperconvex Hartogs domain which is not B-regular and for which the corresponding
hypersurface Green function is discontinuous.  Thus B-regularity is a genuine
boundary regularity hypothesis rather than a reformulation of hyperconvexity.

This difficulty is closely related to a broader circle of
Dirichlet problems with singular boundary data.  Nilsson--Wikstr\"om \cite{NW21} introduced the
pluripotential notion of quasiboundedness and used Jensen measures to study
generalized Dirichlet problems with unbounded boundary values.  Nilsson
\cite{Nilsson22} then proved continuity results for Perron--Bremermann
envelopes with extended-real continuous data on B-regular domains under
additional control of the singular growth, and obtained applications to the
complex Monge--Amp\`ere equation.  Boundary discontinuities concentrated on
small exceptional sets were studied further in \cite{Nilsson25}.  The present
problem is closely connected with this circle of ideas because the bounded
truncations used below are precisely Perron--Bremermann solutions.  However,
the main issue here is more specific: for
$\phi=-\log|f|$ one has to identify the increasing canonical envelopes
$PB_{\phi_j}$ with the desingularized Green function $\widetilde G_A$.  This
is exactly the identification that is missing in the argument of \cite{LS99}.
Instead of imposing an abstract strong-majorant condition, the proof below
uses the geometry of the divisor and analytic discs to recover the logarithmic
mass lost by truncation.  In this sense the argument complements the
Perron--Bremermann theory of \cite{NW21,Nilsson22,Nilsson25}; the more recent
Dirichlet theory on quasi B-regular domains in \cite{DLH26} leads back to the
same Green-function continuity problem from the viewpoint of unbounded
boundary data.  To the best of our knowledge, no general theorem after
\cite{Dieu05} covers a reduced hypersurface with arbitrary singular boundary
points under the assumptions considered here.

After multiplying \(f\) by a nonzero constant, we may suppose throughout
that
\begin{equation}\label{eq:norm}
 \|f\|_{\overline\Omega}\le1,
 \qquad
 \phi:=-\log|f|\ge0.
\end{equation}
By \cite[Proposition~3.2]{LS99},
\[
 \widetilde G_A:=G_{A,\Omega}-\log|f|
\]
extends plurisubharmonically across \(A\) and is locally bounded above there.
Thus the explicit logarithmic singularity is separated from the function
whose regularity will be studied.

We first recall the notion of B-regularity introduced by Sibony
\cite{Sibony87}.
\begin{definition}\label{def:Bregular}
A bounded domain \(\Omega\Subset\C^n\) is called \emph{B-regular} if for every
\(\psi\in C(\partial\Omega)\) there exists
\(u\in\PSH(\Omega)\cap C(\overline\Omega)\) such that
\(u|_{\partial\Omega}=\psi\).
\end{definition}

The corresponding solution of the Dirichlet problem is described by the
Perron--Bremermann envelope.  For \(\psi\in C(\partial\Omega)\), define
\begin{equation}\label{eq:PB-def}
 PB_{\psi}(z)
 :=\sup\Bigl\{u(z):u\in\PSH(\Omega),\
       \limsup_{\Omega\ni w\to\xi}u(w)\le\psi(\xi)
       \text{ for every }\xi\in\partial\Omega\Bigr\}.
\end{equation}
If \(\Omega\) is B-regular, then
\(PB_{\psi}\in\PSH(\Omega)\cap C(\overline\Omega)\), it is maximal in
\(\Omega\), and \(PB_{\psi}=\psi\) on \(\partial\Omega\); see, for example,
\cite[Chapter~4]{Klimek}.  We call \(PB_{\psi}\) the Perron--Bremermann
solution with boundary value \(\psi\).  If \(\psi\) is given on
\(\overline\Omega\), the notation \(PB_{\psi}\) means
\(PB_{\psi|_{\partial\Omega}}\).

For \(t>0\) put
\begin{equation}\label{eq:trunc}
 \eps_t:=e^{-t},\qquad
 \phi_t:=\min\{\phi,t\}
       =-\log\max\{|f|,\eps_t\},
 \qquad
 v_t:=P_\Omega\phi_t,
\end{equation}
where \(P_\Omega\phi_t\) is the largest plurisubharmonic minorant of
\(\phi_t\) on \(\Omega\).  These bounded envelopes will be the main
approximating objects in the proof.

We say that \(\Omega\) is \emph{convexifiable} if
\(\Omega=\Phi(D)\), where \(D\Subset\C^n\) is a bounded convex domain and
\(\Phi\) is biholomorphic between neighborhoods of \(\overline D\) and
\(\overline\Omega\).  The extension across the closures is part of the
definition.

We first record an important geometric consequence of the main SID theorem stated below.
\begin{theorem}\label{thm:convexifiable}
Let \(\Omega\Subset\C^n\), \(n\ge2\), be B-regular and convexifiable, and
let $f$ define a reduced divisor on a neighborhood of $\overline\Omega$, with $A=(f)|_\Omega$.  Then
\[
 G_{A,\Omega}\in C(\Omega,[-\infty,0]).
\]
Moreover, with \(v_j=P_\Omega\min\{-\log|f|,j\}\),
\[
 v_j\uparrow\widetilde G_A\quad\text{on }\Omega\setminus A,
 \qquad
 \left(\sup_jv_j\right)^*=\widetilde G_A\quad\text{on }\Omega.
\]
\end{theorem}

Proposition~\ref{thm:geometry} shows that every convexifiable pair under consideration admits an SID.  Hence Theorem~\ref{thm:convexifiable} follows immediately from Theorem~\ref{thm:main}.

\begin{remark}\label{rem:Bregular-convexifiable}
For a bounded convex domain, Sibony's results imply B-regularity when the boundary contains no nonconstant analytic disc; conversely, B-regularity excludes such discs by the maximum principle.  Thus the two conditions are equivalent in the convex case; see \cite{Sibony87}.  Since B-regularity and boundary analytic discs are preserved by biholomorphisms extending across the closures, the same characterization holds for convexifiable domains.
\end{remark}

The principal result of \cite{Dieu05} treats a finite union of hyperplanes
through a common interior point, under a radial inward condition near the
boundary trace and suitable boundary barriers elsewhere.  Its Hartogs-hull
comparison uses the homogeneity of a defining polynomial.  For a general
hypersurface, the vanishing order along a moving one-variable family can vary,
and no such homogeneous identity is available.

The main theorem below replaces homogeneity by two quantitative ingredients.
On the target side, a complex sector of inward deformations provides both
boundary depth and the Weierstrass variable used at singular boundary points.
On the source side, the proof uses the root-cell construction underlying
\cite[Lemma~3.2 and the proof of Theorem~3.1]{DDF12}.  Since our target is a
domain in $\C^n$, the final Cartan gluing is performed by an explicit additive
Cousin splitting.  This point is useful for the parameter bookkeeping: for
each fixed approximation level, the target zero contours, the source Cartan
pair, and the gluing constant are fixed before the harmonic cutoff and before
the large root-distribution integer.  Consequently the Rouch\'e margins used
to preserve divisor intersections never depend on a parameter chosen later.

Let $\cA(\Omega)$ denote the holomorphic maps from a neighborhood of
$\overline\D$ into $\Omega$, and write
$\cA(\Omega,z_0)=\{h\in\cA(\Omega):h(0)=z_0\}$.

We isolate the geometric condition used below.  Its quantitative depth estimate is deliberately local: only a neighborhood of the boundary trace of the divisor must be pushed a distance comparable with the deformation parameter.

For $0<\theta<\pi/2$ and $t_0>0$ put
\[
 \Gamma_{\theta,t_0}
 :=\{\tau\in\C:0<|\tau|<t_0,\ |\arg\tau|<\theta\}.
\]
\begin{definition}\label{def:sectorial}
Let $f$ define a reduced divisor on a neighborhood of $\overline\Omega$, and put $A=(f)|_\Omega$.  The pair $(\Omega,A)$ is said to admit a
\emph{sectorial inward deformation}, abbreviated \emph{SID}, if there are neighborhoods
$U$ of $\overline\Omega$ and $V$ of $A\cap\partial\Omega$, constants
$0<\theta<\pi/2$, $t_0,c,C>0$, and a holomorphic map
\[
 S:U\times\Delta_{t_0}\longrightarrow\C^n,
 \qquad S_\tau(z):=S(z,\tau),
\]
such that $S_0=\mathrm{id}$ and, for every $\tau\in\Gamma_{\theta,t_0}$,
\begin{align}
 S_\tau(\Omega)&\subset\Omega,\label{eq:sector-selfmap}\\
 \sup_{z\in\overline\Omega}|S_\tau(z)-z|&\le C|\tau|,\label{eq:sector-displacement}\\
 \dist(S_\tau(z),\partial\Omega)&\ge c|\tau|,
 \qquad z\in V\cap\overline\Omega.\label{eq:sector-depth}
\end{align}
In addition, at every singular boundary point of the divisor,
\begin{equation}\label{eq:finite-escape}
 p\in\operatorname{Sing}(A)\cap\partial\Omega
 \quad\Longrightarrow\quad
 \tau\longmapsto f(S_\tau(p))\ \hbox{is not identically zero near }0.
\end{equation}
\end{definition}

Condition \eqref{eq:finite-escape} is independent of the chosen reduced defining function: multiplication of $f$ by a nowhere-vanishing holomorphic factor does not change whether the one-variable germ vanishes identically.  It says only that a singular boundary point is not trapped in the divisor by the whole deformation orbit.

\begin{remark}\label{rem:sector-geometry}
Three features of the SID condition are useful.  First, no point of $\Omega$ is required to be fixed.  A submean argument will select a nearby moving center on a smaller sector.  Second, the depth estimate \eqref{eq:sector-depth} is imposed only near $A\cap\partial\Omega$; singular points of $A$ contained in a compact subset of $\Omega$ are treated by ordinary local Weierstrass preparation.  Third, local invertibility follows automatically from $D_zS_0=I$ on a neighborhood of the compact set $\overline\Omega$: the maps $S_\tau$ are locally biholomorphic near $\overline\Omega$ for all sufficiently small $|\tau|$.  At regular boundary points of $A$ the insertion is transverse to the divisor and does not use \eqref{eq:finite-escape}; that condition is needed only at singular boundary points, where the parameter $\tau$ itself becomes the Weierstrass variable.
\end{remark}

\begin{remark}\label{rem:arkiv-sector}
The relation with \cite[Theorem~2.2]{Dieu05} is complementary rather than one of formal inclusion.  The radial hypothesis there is imposed only on a relative neighborhood of $A\cap\partial\Omega$ and is combined with holomorphic convexity and a boundary regularity condition on the complementary part of $\partial\Omega$; it need not produce a global holomorphic self-map of $\Omega$ as required by SID.  In particular, \cite[Corollary~2.9]{Dieu05} applies to bounded pseudoconvex strongly star-shaped domains and finite unions of complex hyperplanes through the star center, without assuming convexifiability or B-regularity.  Conversely, Theorem~\ref{thm:main} allows a general reduced hypersurface, but at each singular boundary point it requires finite-order escape.  For the homogeneous divisor in \cite{Dieu05}, radial dilation about the common intersection point preserves the divisor; hence at a singular boundary point $p$ one has $f(tp)=0$ for all $t$, and this radial family fails \eqref{eq:finite-escape}.  Thus Theorem~\ref{thm:main} directly recovers the radial situation of \cite{Dieu05} only when that radial motion extends to a global SID with the quantitative depth estimate and the finite-order escape condition is satisfied.  Neither theorem therefore subsumes the other.  On bounded convex B-regular domains, and more generally after a convexifying biholomorphism, Theorem~\ref{thm:convexifiable} nevertheless recovers the corresponding homogeneous hyperplane examples by using the contraction from Proposition~\ref{thm:geometry} toward a point outside $A$, rather than the dilation through the common intersection point.
\end{remark}

We can now state the main theorem of the paper.  Its approximation part does not require B-regularity; B-regularity enters only in the last step, where the bounded canonical envelopes become continuous Perron--Bremermann solutions.
\begin{theorem}\label{thm:main}
Let $\Omega\Subset\C^n$, $n\ge2$, be a domain.  Let $f\not\equiv0$ be holomorphic on an open neighborhood $U_f$ of $\overline\Omega$ and assume that the divisor $(f)$ is reduced on $U_f$.  Suppose that $\{f=0\}\cap\Omega\ne\varnothing$, and put $A=(f)|_\Omega$.  Assume that the pair $(\Omega,A)$ admits an SID in the sense of Definition~\ref{def:sectorial}.  Multiplying $f$ by a nonzero constant if necessary, assume $\|f\|_{\overline\Omega}\le1$.  Then, with
$v_j=P_\Omega\min\{-\log|f|,j\}$,
\begin{equation}\label{eq:SID-limit}
 v_j\nearrow\widetilde G_A\quad\text{on }\Omega\setminus A,
 \qquad
 \left(\sup_jv_j\right)^*=\widetilde G_A\quad\text{on }\Omega.
\end{equation}
If, in addition, $\Omega$ is B-regular, then
\begin{equation}\label{eq:PB-limit}
 PB_{\phi_j}\nearrow\widetilde G_A\quad\text{on }\Omega\setminus A,
 \qquad
 \left(\sup_jPB_{\phi_j}\right)^*=\widetilde G_A,
\end{equation}
and
\[
 G_{A,\Omega}\in C(\Omega,[-\infty,0]).
\]
\end{theorem}

The SID hypothesis is genuinely weaker than convexifiability.  We now give a non-convexifiable Kohn--Nirenberg-type model for which the singular-boundary part of Theorem~\ref{thm:main} is nonvacuous.
\begin{example}\label{ex:KN-sectorial}
Put
\[
 P(z):=|z|^8+\frac{15}{7}|z|^2\operatorname{Re}z^6,
 \qquad
 \rho(z,w):=|w|^2+P(z)+|z|^{10}-1,
\]
and
\[
 \Omega_{\rm KN}:=\{(z,w)\in\C^2:\rho(z,w)<0\}.
\]
The bound $P(z)\ge-\frac87|z|^8$, together with the term $|z|^{10}$, shows that $\Omega_{\rm KN}$ is bounded.  Moreover,
\[
 \rho_{z\bar z}=16|z|^6+15\operatorname{Re}z^6+25|z|^8
 \ge |z|^6+25|z|^8,
 \qquad \rho_{w\bar w}=1,
\]
so $\Omega_{\rm KN}$ is pseudoconvex.  For the real flow of $X=(-z,-4w)$,
\[
 d\rho(X)=-8(\rho+1)-2|z|^{10},
\]
which is strictly negative on $\partial\Omega_{\rm KN}$.  Hence the boundary is smooth; it is real analytic by the defining equation.  Also
\[
 \rho(e^{-s}z,e^{-4s}w)
 =e^{-8s}\rho(z,w)-(1-e^{-8s})
 +(e^{-10s}-e^{-8s})|z|^{10},
 \qquad s>0.
\]
Thus the real flow stays in the domain and tends to the origin, so $\Omega_{\rm KN}$ is connected.  The Levi form is positive definite on the boundary whenever $z\ne0$; hence the weakly pseudoconvex boundary set is contained in
$\{(0,w):|w|=1\}$ and has vanishing two-dimensional Hausdorff measure.
Sibony's criterion therefore gives B-regularity \cite{Sibony87}.

Set
\[
 S_\tau(z,w):=(e^{-\tau}z,e^{-4\tau}w).
\]
A uniform Taylor estimate on a fixed boundary collar, together with compactness of its complement, gives constants $\theta,t_0,c,C>0$ such that, for $\tau\in\Gamma_{\theta,t_0}$,
\[
 S_\tau(\Omega_{\rm KN})\subset\Omega_{\rm KN},\qquad
 \dist(S_\tau(\Omega_{\rm KN}),\partial\Omega_{\rm KN})\ge c|\tau|,
 \qquad
 \|S_\tau-\mathrm{id}\|_{\overline\Omega_{\rm KN}}\le C|\tau|.
\]

At $p=(0,1)$, with $\eta=2(w-1)$, the boundary equation becomes
\[
 \operatorname{Re}\eta+P(z)+|z|^{10}+\frac14|\eta|^2=0.
\]
Its degree-eight tangential part is the classical Kohn--Nirenberg polynomial \cite{KN73}.  For completeness, the local supporting obstruction can be seen directly.  If a smooth holomorphic supporting hypersurface through $p$ existed, its tangent
space would have to be the complex tangent space of $\partial\Omega_{\rm KN}$
at $p$; otherwise the restriction of the defining function to the
hypersurface would change sign.  Hence, in the coordinates $(z,\eta)$, it
would be a graph $\eta=h(z)$ with $h(0)=0$ and $h'(0)=0$.  If the first nonzero term of $h$ has degree smaller than eight, its harmonic real part has a negative direction and dominates the defining equation.  If its degree is greater than eight, or if $h\equiv0$, a negative direction of $P$ already forces the graph into the domain.  In the remaining case
$h(z)=cz^8+O(z^9)$, take
\[
 \theta_k=\frac{\pi}{6}+\frac{k\pi}{3},\qquad 0\le k\le5.
\]
Then $P(re^{i\theta_k})=-\frac87r^8$, while
$\sum_{k=0}^5\operatorname{Re}(ce^{8i\theta_k})=0$; hence for at least one
$k$ the degree-eight contribution from $\operatorname{Re}h$ is nonpositive.
The terms $|z|^{10}$ and $|\eta|^2/4$ have higher order, so the graph again
enters $\Omega_{\rm KN}$, a contradiction.  Thus there is no smooth local
holomorphic supporting hypersurface at $p$.  Since local convexifiability
would pull a complex supporting hyperplane back to such a smooth hypersurface,
$\Omega_{\rm KN}$ is not locally convexifiable at $p$, and hence is not
convexifiable.

The SID theorem also applies to a divisor with a genuine singular boundary point.  For $\lambda>0$ let
\[
 f_\lambda(z,w):=(w-1)^2-\lambda z^2,
 \qquad A_\lambda:=(f_\lambda)|_{\Omega_{\rm KN}}.
\]
The extending zero set is reduced and has its only singular point at $(0,1)$.  It meets the domain, since for sufficiently small $t>0$,
\[
 \rho(t,1-\sqrt\lambda\,t)
 =-2\sqrt\lambda\,t+\lambda t^2+\frac{22}{7}t^8+t^{10}<0.
\]
Finally,
\[
 f_\lambda(S_\tau(0,1))=(e^{-4\tau}-1)^2,
\]
which has order two at $\tau=0$.  Thus $(\Omega_{\rm KN},A_\lambda)$ admits an SID, while $\Omega_{\rm KN}$ is not convexifiable.
\end{example}

\subsection{Disc formulas and bounded truncations}

For an extended real-valued function $\psi$ on $\Omega$, set
\[
 P_\Omega\psi(z)
 :=\sup\{u(z):u\in\PSH(\Omega),\ u\le\psi\ \hbox{on }\Omega\}.
\]
For the divisor $A=(f)$ and $g\in\cA(\Omega)$ with $f\circ g\not\equiv0$, define
\[
 H_A(g):=\sum_{a\in\D:\,f(g(a))=0}\ord_a(f\circ g)\log|a|.
\]
The precise divisor disc formula below, involving the pullback-divisor weight $H_A$, is due to L\'arusson and Sigurdsson \cite[Theorem~3.3]{LS99}.  Drinovec Drnov\v{s}ek and Forstneri\v{c} later gave a new direct treatment of the associated Lelong and reduced Lelong functionals: applied with $\alpha=\nu_A$, \cite[Theorem~3.1]{DDF12} identifies their envelopes with the same extremal Green function $G_{A,\Omega}$.  This is closely related to, but not literally the formula below, since the Lelong functional in \cite{DDF12} is defined from $\alpha(g(a))$ and the multiplicity of the analytic disc at $a$, whereas $H_A$ records the full intersection order $\operatorname{ord}_a(f\circ g)$.  We record the L\'arusson--Sigurdsson divisor formula together with two immediate consequences used repeatedly.

\begin{proposition}\label{prop:obstacle-disc}
Let $\Omega\Subset\C^n$ be a domain.  Then
\begin{equation}\label{eq:obstacle}
 \widetilde G_A=P_\Omega\phi
 =\sup\{u:u\in\PSH(\Omega),\ u\le-\log|f|\}.
\end{equation}
Moreover, for $z\in\Omega\setminus A$,
\begin{align}
 G_{A,\Omega}(z)
 &=\inf_{g\in\cA(\Omega,z)}H_A(g),\label{eq:discG}\\
 \widetilde G_A(z)
 &=\inf_{g\in\cA(\Omega,z)}
   \int_{\T}-\log|f(g(\zeta))|\,d\sigma(\zeta),\label{eq:discU}
\end{align}
where $d\sigma(e^{it})=dt/(2\pi)$.
\end{proposition}

\begin{proof}
Since $G_{A,\Omega}\le0$, we have $\widetilde G_A\le-\log|f|=\phi$.  Conversely, if $u\in\PSH(\Omega)$ and $u\le\phi$, then
\[
 u+\log|f|\le0,
 \qquad
 \nu_{u+\log|f|}\ge\nu_{\log|f|}=\nu_A.
\]
Hence $u+\log|f|\le G_{A,\Omega}$ by \eqref{eq:defG}, and therefore $u\le\widetilde G_A$.  This proves \eqref{eq:obstacle}.

The preceding comparison is first read on $\Omega\setminus A$.  Both $u$ and the
plurisubharmonic extension of $\widetilde G_A$ satisfy
\[
 u(p)=\limsup_{\Omega\setminus A\ni z\to p}u(z),\qquad
 \widetilde G_A(p)=\limsup_{\Omega\setminus A\ni z\to p}\widetilde G_A(z),
 \qquad p\in A,
\]
by the submean property (equivalently, by removability across the analytic
hypersurface).  Hence $u\le\widetilde G_A$ extends across $A$, and so does
\eqref{eq:obstacle}.

Formula \eqref{eq:discG} is the cited theorem.  If $F=f\circ g$, Jensen's formula gives
\[
 \int_{\T}-\log|F|\,d\sigma
 =-\log|F(0)|
 +\sum_{F(a)=0}\ord_a(F)\log|a|.
\]
Here $F(0)\ne0$ and $F$ extends past $\overline\D$.  Its boundary
zeros are finite and cause only integrable logarithmic
singularities.  Jensen's formula follows first when there are
no boundary zeros, and the general case follows by applying
it on circles of radii tending to $1$.  Notice the sign:
$\log|a|<0$ for every zero inside the disc.
Taking the infimum over $g\in\cA(\Omega,z)$ and using \eqref{eq:discG} gives \eqref{eq:discU}.
\end{proof}

The function $\phi_t$ in \eqref{eq:trunc} is finite and continuous.  Since
\[
 -\phi_t=\max\{\log|f|,-t\}\in\PSH(\Omega),
\]
it is also plurisuperharmonic.  Poletsky's theorem gives
\begin{equation}\label{eq:vt-disc}
 v_t(z)
 =\inf_{g\in\cA(\Omega,z)}
 \int_{\T}\phi_t(g(\zeta))\,d\sigma(\zeta);
\end{equation}
see \cite{Poletsky91,Rosay03,DDF12}.  From \eqref{eq:obstacle} we obtain
\begin{equation}\label{eq:vleU}
 0\le v_t\le\widetilde G_A,
 \qquad
 v_s\le v_t\quad(0<s<t).
\end{equation}
Thus only the reverse inequality in the limit remains to be proved.

\begin{definition}\label{def:qub}
Following \cite[Definition 2.4]{DLH26}, with the terminology adapted from
Nilsson--Wikstr\"om \cite{NW21}, a function $u\in\PSH(\Omega)$ is called
quasi upper bounded if there are upper bounded functions
$u_j\in\PSH(\Omega)$ such that $u_j\nearrow u$ outside a pluripolar
subset of $\Omega$.
\end{definition}

The truncations in \eqref{eq:trunc} have an intrinsic interpretation which will be useful later.  Namely, for every $t>0$,
\begin{equation}\label{eq:canonical-trunc}
 v_t=P_\Omega\min\{\widetilde G_A,t\}.
\end{equation}
Indeed, $v_t\le\widetilde G_A$ by \eqref{eq:vleU} and $v_t\le t$, hence $v_t$ is admissible for the envelope on the right-hand side of \eqref{eq:canonical-trunc}.  Conversely,
\[
 \min\{\widetilde G_A,t\}\le \min\{\phi,t\}=\phi_t,
\]
so every plurisubharmonic minorant of $\min\{\widetilde G_A,t\}$ is admissible for $P_\Omega\phi_t=v_t$.  This proves \eqref{eq:canonical-trunc}.

Thus the convergence of $v_j$ identifies the canonical bounded envelopes
of $\widetilde G_A$ and, in the terminology of \cite{DLH26,NW21}, establishes
its quasi upper boundedness.  Theorem~\ref{thm:main} proves this convergence
pointwise off $A$.

The bounded truncations have a simple Perron--Bremermann interpretation.
\begin{proposition}
\label{prop:PB-trunc}
Assume in addition that $\Omega$ is B-regular.  Then, for every $t>0$,
\begin{equation}\label{eq:PB=canonical}
 PB_{\phi_t}=v_t=P_\Omega\phi_t
     =P_\Omega\min\{\widetilde G_A,t\}
 \quad\hbox{on }\Omega.
\end{equation}
In particular, $PB_{\phi_s}\le PB_{\phi_t}$ for $0<s<t$, and
\[
 PB_{\phi_t}|_{\partial\Omega}=\phi_t|_{\partial\Omega}
 \nearrow -\log|f|\big|_{\partial\Omega}
\]
in the extended-real sense.
\end{proposition}

\begin{proof}
By B-regularity, $PB_{\phi_t}$ is the unique continuous maximal
plurisubharmonic function on $\Omega$ with boundary value $\phi_t$.
Since
\[
 -\phi_t=\max\{\log|f|,-t\}\in\PSH(\Omega),
\]
the function $PB_{\phi_t}-\phi_t$ is plurisubharmonic and has boundary
value zero.  Hence $PB_{\phi_t}\le\phi_t$, so
$PB_{\phi_t}\le v_t=P_\Omega\phi_t$.  Conversely, if
$u\in\PSH(\Omega)$ and $u\le\phi_t$, then its boundary upper limit is
at most $\phi_t|_{\partial\Omega}$.  Hence $u$ is one of the competitors
in the boundary envelope \eqref{eq:PB-def}, and therefore
$u\le PB_{\phi_t}$.  Taking the supremum over all such $u$ yields
$v_t\le PB_{\phi_t}$.  Thus $PB_{\phi_t}=v_t$, and
\eqref{eq:canonical-trunc} gives the last equality in
\eqref{eq:PB=canonical}.  The remaining assertions follow from the
monotonicity of $\phi_t$.
\end{proof}

On a B-regular domain, Proposition \ref{prop:PB-trunc} shows that the
canonical envelopes of $\widetilde G_A$ are simultaneously continuous
Perron--Bremermann solutions.  Hence the same sequence links the
disc-functional proof, continuity, the Hartogs-hull construction, and
the comparison principle used in Corollary \ref{cor:Dirichlet}.

\begin{remark}
\label{rem:PB-defect}
We record the possible regularization defect of the Perron--Bremermann limit.  Under the B-regularity assumption in Theorem~\ref{thm:main}, the
pointwise limit
\[
 \sup_{j\ge1}PB_{\phi_j}
\]
is lower semicontinuous because the functions $PB_{\phi_j}$ are
continuous and increasing.  Formula \eqref{eq:PB-limit} says that this
limit already equals $\widetilde G_A$ on $\Omega\setminus A$, while on
$A$ one may have to take the upper regularization:
\[
 \Bigl(\sup_{j\ge1}PB_{\phi_j}\Bigr)^*=\widetilde G_A.
\]
Thus any possible pointwise defect is confined to the pluripolar set
$A$.  At a point $p\in A\cap\partial\Omega$ one has
$\phi_j(p)=j$, so the boundary values tend to $+\infty$; this does not
assert that $\widetilde G_A(z)\to+\infty$ as $z\to p$.  At
$p\in\partial\Omega\setminus A$, the boundary data eventually stabilize
to $-\log|f(p)|$, and Corollary~\ref{cor:Dirichlet} gives the corresponding
boundary limit of $\widetilde G_A$.
\end{remark}

\subsection{A Dirichlet consequence}

A function $u\in\PSH(\Omega)$, $u\not\equiv-\infty$, is maximal if for every $G\Subset\Omega$ and every $v\in\PSH(G)$ satisfying
\[
 \limsup_{G\ni z\to\xi}v(z)\le u(\xi),\qquad \xi\in\partial G,
\]
one has $v\le u$ on $G$.  A set $F\subset\overline\Omega$ is called $\Omega$-pluripolar if there is $w\in\PSH(\Omega)$, $w<0$, $w\not\equiv-\infty$, such that the boundary upper regularization $w^*$ equals $-\infty$ on $F$.  In particular, $A\cap\partial\Omega$ is $\Omega$-pluripolar: by \eqref{eq:norm},
\[
 w:=\log|f|-1<0
\]
is plurisubharmonic, is not identically $-\infty$, and satisfies $w^*=-\infty$ at every point of $A\cap\partial\Omega$.

\begin{corollary}\label{cor:Dirichlet}
Under the hypotheses of Theorem \ref{thm:convexifiable}, the function $\widetilde G_A$ is locally bounded and maximal on $\Omega$, and
\begin{equation}\label{eq:MAzero}
 (\ddc\widetilde G_A)^n=0
\end{equation}
in the Bedford--Taylor sense.  Moreover,
\begin{equation}\label{eq:boundaryphi}
 \lim_{\Omega\ni z\to p}\widetilde G_A(z)=-\log|f(p)|,
 \qquad p\in\partial\Omega\setminus A.
\end{equation}
Among quasi upper bounded maximal plurisubharmonic functions on $\Omega$ which are bounded from below and have the boundary values \eqref{eq:boundaryphi} outside an $\Omega$-pluripolar set, $\widetilde G_A$ is unique.
\end{corollary}

\begin{proof}
By Proposition \ref{prop:obstacle-disc}, $\widetilde G_A\ge0$, and it is locally bounded above near $A$ by \cite[Proposition 3.2]{LS99}.  On $\Omega\setminus A$, the function $G_{A,\Omega}$ is maximal by \cite[Theorem 3.3]{LS99}, while $\log|f|$ is pluriharmonic.  Hence $\widetilde G_A$ is maximal on $\Omega\setminus A$.  A locally bounded plurisubharmonic function which is maximal outside a pluripolar set is maximal everywhere; see \cite[Chapter 4]{Klimek}.  This gives \eqref{eq:MAzero}.

Let $p\in\partial\Omega\setminus A$.  Since $\widetilde G_A\le\phi$,
\[
 \limsup_{z\to p}\widetilde G_A(z)\le\phi(p).
\]
For every sufficiently large $j$, $\phi_j(p)=\phi(p)$, and B-regularity gives
\[
 \lim_{z\to p}v_j(z)=\phi_j(p)=\phi(p).
\]
Since $v_j\le\widetilde G_A$, \eqref{eq:boundaryphi} follows.

Theorem~\ref{thm:main} and Definition~\ref{def:qub} show that $\widetilde G_A$ is quasi upper bounded.  To verify the uniqueness statement, let $u$ and $v$ be two functions in the stated class.  Choose $\Omega$-pluripolar sets $E_u,E_v\subset\partial\Omega$ outside which $u$ and $v$, respectively, have the boundary values \eqref{eq:boundaryphi}, and set
\[
 E:=(A\cap\partial\Omega)\cup E_u\cup E_v.
\]
By \cite[Lemma 2.3]{DLH26}, the finite union $E$ is again $\Omega$-pluripolar.  Since both $u$ and $v$ are bounded from below, their boundary lower regularizations are everywhere greater than $-\infty$.  On $\partial\Omega\setminus E$ the two boundary limits agree, and hence $u^*\le v_*$ there.  Applying \cite[Lemma 2.6(i)]{DLH26} with $u$ as the quasi upper bounded function and $v$ as the maximal comparison function gives $u\le v$ on $\Omega$.  Interchanging the two functions gives the reverse inequality, and therefore $u=v$.
\end{proof}

\section{Sectorial contractions and the convex model}\label{sec:geometry}

The affine contraction toward a point outside the divisor has a small complex sector on which it remains uniformly inward.  This observation places convexifiable domains inside the framework of Definition~\ref{def:sectorial} and also explains the finite-order escape condition.

\begin{proposition}\label{thm:geometry}
Let $D\Subset\C^n$, $n\ge2$, be a bounded convex domain and let
$\Phi$ be biholomorphic from a neighborhood of $\overline D$ onto a
neighborhood of $\overline\Omega$, where $\Omega=\Phi(D)$.  For every reduced divisor $A=(f)|_\Omega$, the pair $(\Omega,A)$ admits an SID.  Moreover, $D$ is B-regular if and only if $\Omega$ is B-regular.
\end{proposition}

\begin{proof}
Choose $z_*\in\Omega\setminus A$, put $w_*:=\Phi^{-1}(z_*)$, and write
\[
 d:=\dist(w_*,\partial D)>0,
 \qquad
 R:=\sup_{w\in D}|w-w_*|.
\]
In convex coordinates set
\[
 \mathcal S_\tau(w):=(1-\tau)w+\tau w_*.
\]
For real $0<t<1$, convexity gives
\[
 B(\mathcal S_t(w),td)\subset D,
 \qquad w\in D.
\]
Write $\tau=t+is$.  Then
\[
 |\mathcal S_\tau(w)-\mathcal S_t(w)|\le |s|R.
\]
Choose a fixed angle $0<\theta<\pi/2$ so small that
$R\tan\theta<d/2$.  If $|\arg\tau|<\theta$ and
$t=\operatorname{Re}\tau>0$, then
$|s|R<td/2$.  Hence
\[
 \mathcal S_\tau(w)\in
 B(\mathcal S_t(w),td)\subset D.
\]
Moreover, using the same interior ball and the $1$-Lipschitz property of the
distance function,
\[
 \dist(\mathcal S_\tau(w),\partial D)
 \ge td-|s|R\ge \frac d2t\ge c_0|\tau|
\]
with a constant $c_0>0$ depending only on $d$, $R$, and $\theta$.  Thus
$\mathcal S_\tau(D)\subset D$ throughout this sector (after restricting to
$|\tau|<t_0$ only to remain inside the fixed neighborhood on which the
transport by $\Phi$ is defined).  Also
\[
 \sup_{w\in D}|\mathcal S_\tau(w)-w|\le R|\tau|.
\]
Choose a neighborhood $U_D$ of $\overline D$ compactly contained in the domain of $\Phi$ and so small that, after decreasing $t_0$, $\mathcal S_\tau(U_D)$ remains in the domain of $\Phi$ for every $|\tau|<t_0$.  Transporting the family by $\Phi$ gives a holomorphic map on a full parameter disc,
\[
 S_\tau:=\Phi\circ\mathcal S_\tau\circ\Phi^{-1}
 \quad\text{on }U:=\Phi(U_D),\qquad |\tau|<t_0.
\]
Because $\Phi$ and $\Phi^{-1}$ extend holomorphically across the closures, the preceding inclusion, depth, and displacement estimates persist on $\Omega$ after changing the constants.

It remains to check finite-order escape at a singular boundary point $p\in A\cap\partial\Omega$.  Put $x=\Phi^{-1}(p)$ and $F=f\circ\Phi$.  If
\[
 \tau\longmapsto F((1-\tau)x+\tau w_*)
\]
vanished identically near $0$, the identity theorem, applied on a connected complex neighborhood of the real segment $[0,1]$ on which the expression is defined, would give $F(w_*)=0$.  This contradicts $z_*\notin A$.  Hence Definition~\ref{def:sectorial} is satisfied.

Finally, B-regularity is invariant under a biholomorphism extending across the closures.  Indeed, continuous boundary data pull back under $\Phi$, a continuous plurisubharmonic extension on one domain composes with the inverse biholomorphism to give such an extension on the other, and the converse follows by symmetry.
\end{proof}

\begin{proof}[Proof of Theorem~\ref{thm:convexifiable}]
By Proposition~\ref{thm:geometry}, every reduced divisor on a convexifiable
domain gives a pair $(\Omega,A)$ admitting an SID.  The approximation
statements and, under B-regularity, the extended-real continuity of
$G_{A,\Omega}$ now follow directly from Theorem~\ref{thm:main}.
\end{proof}

The next example shows that convexifiability is strictly weaker than convexity.
\begin{example}
\label{ex:nonconvex-convexifiable}
In $\C^2$ let
\[
 \Phi(w_1,w_2):=(w_1,w_2+2w_1^2),\qquad
 \Omega:=\Phi(B^2).
\]
Since $\Phi$ is a polynomial automorphism with inverse
$\Psi(z_1,z_2)=(z_1,z_2-2z_1^2)$, the domain is explicitly
\[
 \Omega=\bigl\{(z_1,z_2)\in\C^2:
 |z_1|^2+|z_2-2z_1^2|^2<1\bigr\}.
\]
Thus $\Omega$ is convexifiable, and it is B-regular because the unit
ball is B-regular and B-regularity is preserved by the biholomorphism
in Proposition~\ref{thm:geometry}.  It is nevertheless not convex.  Indeed,
for $r=4/5$ the two points
\[
 (r,2r^2)=\Phi(r,0),\qquad (-r,2r^2)=\Phi(-r,0)
\]
belong to $\Omega$, whereas their midpoint $(0,2r^2)$ does not, since
$|2r^2|>1$.

For a divisor to which Theorem~\ref{thm:convexifiable} applies, one may take
\[
 f(z)=c\,z_1(z_2-2z_1^2),
\]
where $c>0$ is chosen so that $\|f\|_{\overline\Omega}\le1$.  This is a
reduced divisor with an interior node at the origin.
\end{example}

For a special homogeneous divisor, the older Hartogs-hull method still gives a conclusion beyond the B-regular setting. If the convex model $D$ below is B-regular, then Proposition~\ref{prop:boundary-hull} is already an immediate consequence of Theorem~\ref{thm:convexifiable} (equivalently, of Proposition~\ref{thm:geometry} and Theorem~\ref{thm:main}), and Theorem~\ref{thm:convexifiable} in fact allows an arbitrary reduced divisor. The point of the proposition is therefore the complementary case in which $D$ need not be B-regular.  Even in that case Proposition~\ref{thm:geometry} and the approximation part of Theorem~\ref{thm:main} still give the canonical convergence $v_j\nearrow\widetilde G_A$ off $A$, but they do not by themselves give continuity of $G_{A,\Omega}$.  For instance, bounded convex Reinhardt domains such as the polydisc have analytic discs in the boundary and hence are not B-regular, whereas the homogeneous Hartogs-hull argument below still yields continuity for the stated divisor.
\begin{proposition}
\label{prop:boundary-hull}
Let $D\Subset\C^n$ be a bounded convex domain, let $a\in D$, and let
$L_1,\ldots,L_m$ be nonzero, pairwise nonproportional complex linear
forms.  Suppose that $\Omega=\Phi(D)$, where $\Phi$ is biholomorphic
between neighborhoods of $\overline D$ and $\overline\Omega$.  Define
the reduced divisor
\[
 A_D:=\left(\prod_{\nu=1}^m L_\nu(w-a)\right)\Big|_D,
 \qquad A:=\Phi_*A_D.
\]
Then
\[
 G_{A,\Omega}\in C(\Omega,[-\infty,0]).
\]
The zero set extending $A$ across $\partial\Omega$ may be singular at
boundary points.
\end{proposition}

\begin{proof}
Translate $a$ to the origin.  The bounded convex domain $D-a$ is
pseudoconvex and strongly star-shaped about $0$: if $0\le t<1$,
convexity and $a\in D$ give the stronger inclusion
\[
 t\,\overline{(D-a)}\subset D-a.
\]
The divisor becomes a finite union of complex hyperplanes through $0$.
Therefore \cite[Corollary 2.9 and the remark following it]{Dieu05}
gives continuity of its Green function on $D-a$; this is precisely the
case proved there by the Hartogs-hull method.

Finally, composition with $\Phi$ gives a bijection between the negative
plurisubharmonic competitors in \eqref{eq:defG} on $D$ and on $\Omega$.
Lelong numbers and divisor multiplicities are preserved by a
biholomorphism.  Consequently
\[
 G_{A,\Omega}\circ\Phi=G_{A_D,D},
\]
and the claimed extended-real continuity follows.
\end{proof}

The following example has a genuine singular locus whose closure meets the boundary.
\begin{example}
\label{ex:boundary-singularity}
Take $D=B^3\subset\C^3$, $\Phi=\mathrm{id}$, and
\[
 f(z)=z_1z_2.
\]
The divisor is reduced and its singular locus is
\[
 \{z_1=z_2=0\}\cap B^3.
\]
Its closure contains $p=(0,0,1)\in\partial B^3$, where $f(p)=0$ and
$df(p)=0$.  Nevertheless $G_{(f),B^3}$ is continuous in the
extended-real sense by Proposition \ref{prop:boundary-hull}.
\end{example}

\begin{remark}
Theorem~\ref{thm:convexifiable} does not depend on whether the chosen global
defining function has a critical zero on the boundary.  For example, on
$B^2$ the function $f=(1-z_1)z_2$ defines the same divisor in the ball as
$z_2$, although $df(1,0)=0$.  More importantly,
Example~\ref{ex:boundary-singularity} has a genuine intersection of
components whose singular locus reaches the boundary; it is covered
directly by Theorem~\ref{thm:convexifiable}, while
Proposition~\ref{prop:boundary-hull} records the Hartogs-hull
explanation from \cite{Dieu05} for that special homogeneous geometry.
\end{remark}

The SID condition also occurs naturally on Reinhardt domains. We use the
characterization of Dieu--Dung--Hung \cite[Proposition~3.1]{DDH05}: a bounded
Reinhardt domain is B-regular if and only if it is hyperconvex and its
boundary contains no nonconstant analytic disc.

\begin{proposition}\label{prop:Reinhardt-SID}
Let $\Omega\Subset\C^n$ be a bounded Reinhardt domain. Assume that
$\Omega=\{\rho<0\}$ near $\partial\Omega$, where $\rho$ is a $C^1$
defining function satisfying
\[
 d\rho_z(-z)\le -c_0<0,\qquad z\in\partial\Omega.
\]
Assume also that $\Omega$ is hyperconvex and that $\partial\Omega$ contains
no nonconstant analytic disc. Let $f$ be holomorphic and define a reduced divisor on a neighborhood of $\overline\Omega$, put $A=(f)|_\Omega$, and assume $A\ne\varnothing$, and suppose that
\[
 \tau\longmapsto f(e^{-\tau}p)
 \quad\hbox{is not identically zero near }0
\]
for every $p\in\operatorname{Sing}(A)\cap\partial\Omega$. Then
$(\Omega,A)$ admits an SID via
\[
 S_\tau(z):=e^{-\tau}z,
\]
and
\[
 G_{A,\Omega}\in C(\Omega,[-\infty,0]).
\]
\end{proposition}

\begin{proof}
By \cite[Proposition~3.1]{DDH05}, the Reinhardt assumptions imply that
$\Omega$ is B-regular. By continuity of $d\rho_z(-z)$, after shrinking to a
fixed boundary collar one has
\[
 \frac{d}{dt}\rho(e^{-t}z)\le -\frac{c_0}{2}
\]
for all sufficiently small $t>0$ as long as the orbit remains in that collar.
Hence the real contractions $S_t(z)=e^{-t}z$ map $\Omega$ into itself for
all sufficiently small $t>0$.  If $L$ bounds $|\nabla\rho|$ on the chosen
collar, the defining-function decrease and
$|\rho(q)|\le L\,\dist(q,\partial\Omega)$ show that points coming from the
collar have boundary depth at least $c_1t$ for some $c_1>0$; points outside
the collar remain a fixed positive distance from the boundary for small $t$.
Because $\Omega$ is Reinhardt, multiplication by $e^{-is}$ preserves
$\Omega$ and its Euclidean boundary distance. Thus for
$\tau=t+is$ in a sufficiently small fixed sector,
$S_\tau=e^{-is}S_t$ satisfies the self-map and depth estimates of
Definition~\ref{def:sectorial}. The displacement estimate follows from
$|e^{-\tau}-1|=O(|\tau|)$ on the bounded set $\overline\Omega$. The final
assumption is exactly the finite-order escape condition. The conclusion now
follows from Theorem~\ref{thm:main}.
\end{proof}

A basic family is provided by the Reinhardt ellipsoids
\[
 \Omega_{\mathbf a}:=\left\{z\in\C^n:
 \sum_{j=1}^n |z_j|^{a_j}<1\right\},\qquad a_j>1.
\]
Their B-regularity is also recorded explicitly in \cite[Corollary~3.4]{DDH05}.
For the defining function
$\rho_{\mathbf a}=\sum_j|z_j|^{a_j}-1$ one has on the boundary
\[
 d\rho_{\mathbf a,z}(-z)
 =-\sum_{j=1}^n a_j|z_j|^{a_j}
 \le-\min_j a_j.
\]
Hence scalar contraction gives an SID for every reduced divisor satisfying the
finite-order escape condition.  Smooth divisors require no additional escape
hypothesis, while singular examples can be obtained, for $n\ge2$, from
$(z_1-1)^2-\mu z_2^2$ with $\mu>0$.

The preceding criterion also gives explicit nonconvex Reinhardt examples.
\begin{example}\label{ex:Reinhardt-SID}
For $\lambda>0$ put
\[
 \Omega_\lambda:=\bigl\{(z_1,z_2)\in\C^2:
 |z_1|^2+|z_2|^2+\lambda|z_1z_2|^2<1\bigr\}.
\]
With
\[
 \rho_\lambda(z):=|z_1|^2+|z_2|^2+\lambda|z_1z_2|^2-1,
\]
the Levi matrix is
\[
 \begin{pmatrix}
 1+\lambda|z_2|^2 & \lambda\overline z_1z_2\\
 \lambda z_1\overline z_2 & 1+\lambda|z_1|^2
 \end{pmatrix},
 \qquad
 \det=1+\lambda(|z_1|^2+|z_2|^2)>0.
\]
Thus $\rho_\lambda$ is strictly plurisubharmonic. In particular
$\Omega_\lambda$ is hyperconvex and its boundary contains no nonconstant
analytic disc, so it is B-regular also by \cite[Proposition~3.1]{DDH05}.
On $\partial\Omega_\lambda$,
\[
 d\rho_\lambda(-z)
 =-2(|z_1|^2+|z_2|^2)-4\lambda|z_1z_2|^2
 =-2\bigl(1+\lambda|z_1z_2|^2\bigr)\le-2.
\]
Hence Proposition~\ref{prop:Reinhardt-SID} applies.

The domain need not be convex. Indeed, if $\lambda>3$, choose $r>0$ by
$2r^2+\lambda r^4=1$ and put $q:=\lambda r^2$.  Since the boundary equation gives
$\lambda=q^2+2q$, the inequality $\lambda>3$ implies $q>1$.  For small
$\varepsilon>0$,
\[
 \rho_\lambda(r+\varepsilon,r-\varepsilon)
 =2(1-q)\varepsilon^2+\lambda\varepsilon^4<0,
\]
and the same holds with the two coordinates interchanged, whereas their
midpoint $(r,r)$ lies on the boundary.

Finally, for $\mu>0$ define
\[
 f_\mu(z_1,z_2):=(z_1-1)^2-\mu z_2^2,
 \qquad A_\mu:=(f_\mu)|_{\Omega_\lambda}.
\]
The divisor is reduced and has the unique singular point $p=(1,0)$, which
lies on $\partial\Omega_\lambda$. Its zero set meets $\Omega_\lambda$:
for small $t>0$ the point $(1-\sqrt\mu\,t,t)$ lies on one component of
$\{f_\mu=0\}$ and satisfies $\rho_\lambda<0$. Moreover,
\[
 f_\mu(e^{-\tau},0)=(e^{-\tau}-1)^2,
\]
so the finite-order escape condition holds with order two. Therefore
$G_{A_\mu,\Omega_\lambda}$ is continuous in the extended-real sense.
\end{example}

\section{Root selection and analytic-disc gluing}\label{sec:insertion}

This section contains the source-side part of the proof.  The target-side
insertions adapted to SID are postponed to Section~\ref{sec:sectorial-proof}.
The only input taken from the source construction of
Drinovec Drnov\v{s}ek--Forstneri\v{c} is the root-cell mechanism in the proof
of \cite[Lemma~3.2 and Theorem~3.1]{DDF12}.  We isolate below exactly the
selection statement that is needed.  The subsequent gluing is proved directly
in the Euclidean target $\C^n$, so its smallness threshold is independent of
the later large root-distribution integer.

For a measurable arc $J\subset\T$ we write $|J|:=\sigma(J)$ for its
normalized arc length.  The following selection form is extracted from the
root-cell construction in the proof of \cite[Lemma~3.2]{DDF12}.  It is stated
only in the form used later.

\begin{lemma}\label{lem:root-selection}
Let $\Delta\subset\D$ be a simply connected smoothly bounded source domain
whose boundary contains an open arc $J'\subset\T$, and assume
$0\notin\overline\Delta$.  Let $J\Subset J'$ be a closed arc of positive length, let
$\chi$ be holomorphic and nowhere zero on $\Delta$, continuous up to a
neighborhood of $J'$ in $\overline\Delta$, and assume $|\chi|=1$ on $J'$.
Fix $\beta\in\D^*$.  Given $\eta>0$, after choosing an arbitrarily thin
one-sided collar of $J$ in $\Delta$, for every sufficiently large integer
$N$ one can select solutions
\[
 z_{N,1},\ldots,z_{N,\ell_N}
\]
of $z^N\chi(z)=\beta$ in that collar such that
\begin{equation}\label{eq:root-selection-count}
 \ell_N\le 2N|J|
\end{equation}
and
\begin{equation}\label{eq:root-selection-weight}
 \sum_{\ell=1}^{\ell_N}\log|z_{N,\ell}|
 \le |J|\log|\beta|+\eta.
\end{equation}
The selected solutions may be taken in pairwise disjoint root cells of the
construction in the proof of \cite[Lemma~3.2]{DDF12}.
\end{lemma}

\begin{proof}
Put $\zeta:=\beta/\chi$.  Thus $\zeta$ is holomorphic and nowhere zero on
$\Delta$, extends continuously to a neighborhood of $J'$ in
$\overline\Delta$, and $|\zeta|=|\beta|<1$ on $J'$.  Fix first the prescribed
one-sided collar $\mathcal C$ of $J$ in $\overline\Delta$.  Choose finitely
many pairwise disjoint closed subarcs
$J_1,\ldots,J_q\Subset J$, separated by short gaps, whose total length
$L$ satisfies
\[
 (|J|-L)|\log|\beta||<\eta/4.
\]
Choose first
\[
 0<\delta<\min\{-\tfrac12\log|\beta|,\eta/(4L)\},
\]
and choose $\epsilon_0>0$ so small that the disc of radius
$\epsilon_0|\beta|$ centered at any point of modulus $|\beta|$ is compactly
contained in $\D^*$.  After making the subarcs shorter if necessary, choose
points $\xi_\nu\in J_\nu$ and pairwise disjoint open neighborhoods
$U_\nu\subset\C$ of $J_\nu$ so small that
$U_\nu\cap\Delta\subset\mathcal C$ and
$\overline{U_\nu}\cap\T\Subset J$,
\begin{equation}\label{eq:zeta-nearly-constant}
 |\zeta(z)-a_\nu|<\epsilon_0|a_\nu|,
 \qquad z\in U_\nu\cap\Delta,
 \qquad a_\nu:=\zeta(\xi_\nu),
\end{equation}
and
\[
 |\log|\chi(z)||<\delta,
 \qquad z\in U_\nu\cap\Delta.
\]
This is exactly the local subdivision used in the proof of
\cite[Lemma~3.2]{DDF12}.

For each $\nu$ let $B_\nu$ be the disc centered at $a_\nu$ with radius
$\epsilon_0|a_\nu|$.  The inverse image of $B_\nu$ under $z\mapsto z^N$ is
the disjoint union of $N$ simply connected root cells.  As $N\to\infty$
these cells approach $\T$ and become equidistributed in the angular
variable.  Hence, for all sufficiently large $N$, at least
$\lceil N|J_\nu|\rceil$ of them are compactly contained in
$U_\nu\cap\Delta$.  On the boundary of any such cell the functions
$z^N-a_\nu$ and $\zeta(z)-a_\nu$ satisfy the strict Rouch\'e inequality by
\eqref{eq:zeta-nearly-constant}.  Therefore every selected cell contains
exactly one solution of $z^N=\zeta(z)$; this is precisely the root-cell
argument in the proof of \cite[Lemma~3.2]{DDF12}.

Every selected root satisfies
\[
 \log|z|=\frac1N\bigl(\log|\beta|-\log|\chi(z)|\bigr)
 \le \frac{\log|\beta|+\delta}{N}.
\]
Select exactly $\lceil N|J_\nu|\rceil$ root cells over each $J_\nu$.
The total number of selected roots is at most $NL+q$, hence at most
$2N|J|$ for all sufficiently large $N$.  By the choice of $\delta$,
$\log|\beta|+\delta<0$.  Since the number of selected roots is at
least $NL$, the preceding pointwise upper bound, whose right-hand side is
negative, gives
\[
 \sum\log|z|
 \le L\bigl(\log|\beta|+\delta\bigr)
 \le |J|\log|\beta|+\frac{\eta}{2}.
\]
Here the last inequality uses the estimate for the omitted gaps and
$L\delta<\eta/4$.  This proves
\eqref{eq:root-selection-weight}.
\end{proof}

For the Rouch\'e argument used below, it is useful to retain not only the
selected source roots from Lemma~\ref{lem:root-selection}, but also small
source discs around them.  The following elementary localization is the local
form of the root cells appearing in the proof of \cite[Lemma~3.2]{DDF12}.

\begin{lemma}\label{lem:source-preimage}
Let $J\Subset\T$ be a closed arc and let $\Delta\subset\D$ be a simply connected smoothly bounded
source domain whose boundary contains an open arc $J'\supset J$, with
$0\notin\overline\Delta$.  Let $\chi$ be holomorphic and nowhere zero on
$\Delta$, of class $C^1$ up to a neighborhood of $J$ in
$\overline\Delta$, and assume $|\chi|=1$ on $J'$.  Fix
$\beta\in\D^*$ and an open disc $B\Subset\D^*$ containing $\beta$.
Assume that a branch of $\log(\xi/\beta)$ is defined on $B$ and that
\[
 \sup_{\xi\in B}|\log(\xi/\beta)|
 <-\frac12\log|\beta|.
\]
This is a smallness condition on $B$ depending only on $\beta$.
For
\[
 a_N(z):=z^N\chi(z),
\]
let $z_N\in\Delta$ be solutions of $a_N(z_N)=\beta$ which converge to
$J$ as $N\to\infty$.  For all sufficiently large $N$ the connected
component $C_N$ of $a_N^{-1}(B)$ containing $z_N$ satisfies
$C_N\Subset\Delta$ and
\[
 a_N:C_N\longrightarrow B
\]
is biholomorphic.  In particular, components corresponding to distinct
solutions of $a_N(z)=\beta$ are pairwise disjoint.  Moreover,
\[
 \operatorname{diam} C_N=O(N^{-1})
\]
uniformly for $z_N$ in a fixed thin collar of $J$.  More precisely, set
\[
 \omega(B):=\sup_{\xi\in B}
 \bigl|\log|\xi|-\log|\beta|\bigr|.
\]
Then
\begin{equation}\label{eq:source-cell-osc}
 N\sup_{z\in C_N}
 \bigl|\log|z|-\log|z_N|\bigr|
 \le \omega(B)+o(1),
\end{equation}
where the $o(1)$ is uniform for $z_N$ in the fixed collar and
$\omega(B)\to0$ as $B$ shrinks to $\{\beta\}$.  More generally, the
biholomorphic-component assertion and the estimates are uniform for any
family of solutions contained in a fixed sufficiently thin collar whose
boundary trace is compactly contained in $J'$ and whose distance to $\T$
tends to zero uniformly as $N\to\infty$.
\end{lemma}

\begin{proof}
Choose a thin simply connected one-sided neighborhood $W\subset\Delta$ of
$J$ on which branches of $\log z$ and $\log\chi$ are fixed.  The $C^1$
boundary regularity of $\chi$ gives uniform bounds for these branches and
their first derivatives on $W$.  In particular, after fixing a
slightly larger collar, there is a constant $C_\chi$ such that
$|\chi'/\chi|\le C_\chi$ there.  This constant is independent of $N$
and of the selected root.  Since $z_N\to J$, all $z_N$ belong to $W$
for large $N$.  Put
\[
 z=z_Ne^{w/N}.
\]
Using $a_N(z_N)=\beta$, we obtain
\[
 \log\frac{a_N(z_Ne^{w/N})}{\beta}
 =w+
 \log\frac{\chi(z_Ne^{w/N})}{\chi(z_N)}.
\]
By the stated smallness condition on $B$, and since $J\Subset J'$, the
corresponding points $z_Ne^{w/N}$ remain on the inner side of $J'$ for $w$
in a fixed bounded simply connected domain $W_0$ containing
\[
 L_B:=\{\log(\xi/\beta):\xi\in\overline B\}
\]
compactly.  On $W_0$ put
\[
 G_N(w):=w+
 \log\frac{\chi(z_Ne^{w/N})}{\chi(z_N)}.
\]
On the fixed bounded set $W_0$, the second term and its first derivative
are $O(N^{-1})$, with constants depending only on the collar, $W_0$, and
$C_\chi$.  The estimates are therefore uniform over all the selected roots,
even when their number grows with $N$.
Choose a slightly smaller domain $W_1$ with
$L_B\Subset W_1\Subset W_0$.  For all large $N$, the strict Rouch\'e
inequality on $\partial W_1$ shows that, for every
$\lambda\in L_B$, the equation $G_N(w)=\lambda$ has exactly one solution in
$W_1$; moreover $G_N'(w)\ne0$ there.  Hence these solutions form a
holomorphic inverse branch over $B$.  The uniform $C^1$ convergence
$G_N\to\mathrm{id}$ yields
\[
 w=\log(\xi/\beta)+o(1)
\]
uniformly in $\xi\in B$.  Returning to $z=z_Ne^{w/N}$ gives
$\operatorname{diam}C_N=O(N^{-1})$, and taking real parts in
$\log z-\log z_N=w/N$ gives \eqref{eq:source-cell-osc}.

To identify the full component, choose a slightly larger disc $B^+$ with
$\overline B\Subset B^+\Subset\D^*$ so close to $B$ that the same strict
logarithmic smallness condition used above remains valid on $B^+$.  Carry out
the same local inverse construction over $B^+$.  The inverse branch $\iota_N$ is
holomorphic on $B^+$, takes its values in $\Delta$, and satisfies
$a_N\circ\iota_N=\mathrm{id}$.  Hence
$\iota_N(\overline B)\Subset\Delta$, and the boundary of
$\iota_N(B)$ maps to $\partial B$.  The set $\iota_N(B)$ is both open and
relatively closed in the component of $a_N^{-1}(B)$ containing $z_N$;
therefore it is that component.  Distinct components over $\beta$ are
disjoint.  This argument uses only the local inverse and does not require a
collar whose side boundary stays a fixed distance from $\T$.
\end{proof}

The gluing argument below is formulated for Cartan pairs.  We recall
the standard notion from \cite[Definition~5.7.1(I)]{Forstneric}.
\begin{definition}\label{def:cartan-pair}
Let $S$ be a complex manifold.  A pair $(K_0,K_1)$ of compact subsets of $S$
is called a \emph{Cartan pair} if $K_0$, $K_1$, $K_0\cup K_1$, and
$K_0\cap K_1$ are Stein compacta and
\[
 \overline{K_0\setminus K_1}\cap
 \overline{K_1\setminus K_0}=\varnothing.
\]
Here a Stein compact is a compact set admitting a basis of open Stein
neighborhoods.  We also use the same terminology when one of the two
source sets has finitely many components: a pair $(D_0,D_1)$ of
relatively compact smoothly bounded open sets in the source plane is
called a \emph{smoothly bounded Cartan pair} when
$(\overline D_0,\overline D_1)$ is a Cartan pair in this sense.  In
the construction below, $D_0$ is connected and $D_1$ is a finite
disjoint union of source discs.  The source pair is chosen with the
additional separation
\begin{equation}\label{eq:source-separation}
 \overline{D_0\setminus D_1}\cap
 \overline{D_1\setminus D_0}=\varnothing;
\end{equation}
this is the separation imposed in
\cite[proof of Theorem~3.1]{DDF12}.
\end{definition}

The source construction in \cite{DDF12} glues sprays over such a pair.
For the present problem, whose target is a domain in $\C^n$, a direct
additive Cousin splitting gives the needed result and its norm estimate.
This avoids a dependence of the gluing threshold on the later choice
of the integer $N$.  The central disc is used on $D_0$, and the
prescribed insertion map is used on $D_1$; their difference is made
small on the overlap.  The correction can be required to vanish at
$0\in D_0\setminus\overline{D_1}$.

\begin{lemma}\label{lem:cartan-additive}
Let $D_0,D_1$ be a smoothly bounded Cartan pair in the
source plane satisfying \eqref{eq:source-separation},
$0\in D_0\setminus\overline{D_1}$, and
$D_0\cup D_1=\D$.  Suppose that
$h$ is holomorphic on a neighborhood of $\overline\D$ and that
$g:D_1\to\C^n$ is holomorphic and continuous on $\overline{D_1}$.
There is a constant $C>0$ depending only on the fixed pair such that,
if
\[
 \|g-h\|_{\overline{D_0\cap D_1}}\le\epsilon,
\]
then there is $H\in\mathcal O(\D,\C^n)\cap
C(\overline\D,\C^n)$ satisfying
\begin{equation}\label{eq:additive-Cartan-estimate}
 H(0)=h(0),\qquad
 \|H-h\|_{\overline{D_0}}+
 \|H-g\|_{\overline{D_1}}\le C\epsilon.
\end{equation}
If both $h(\overline{D_0})$ and $g(\overline{D_1})$
lie in $\Omega$, the same construction gives $H(\overline\D)
\subset\Omega$ whenever $C\epsilon$ is smaller than their
distances from $\partial\Omega$.
\end{lemma}

\begin{proof}
Write $\delta=g-h$ on $D_0\cap D_1$.  By
\eqref{eq:source-separation}, there is a fixed smooth function
$\lambda$ on the source plane such that $0\le\lambda\le1$,
$\lambda=0$ near $\overline{D_0\setminus D_1}$, and
$\lambda=1$ near $\overline{D_1\setminus D_0}$.
Thus the derivatives of $\lambda$ within $D_0\cup D_1$ occur
only in the overlap, including when $D_1$ has several components.
On $D_0\cap D_1$ put
\[
 v=(\bar\partial\lambda)\delta,
\]
and extend $v$ by zero outside the overlap.  This is a bounded,
compactly supported $(0,1)$-form on the source plane, and
$\|v\|_\infty\le\|\bar\partial\lambda\|_\infty\epsilon$.
Its Cauchy--Green transform
\[
 u(z)=\frac{1}{\pi}\int_{\C}
         \frac{v(\xi)}{z-\xi}\,dA(\xi)
\]
is continuous, satisfies $\bar\partial u=v$ in the distribution
sense, and obeys $\|u\|_{L^\infty(\C)}\le C_1\epsilon$.
The integral estimate uses only the fixed bounded support of $v$
and the local integrability of $|z-\xi|^{-1}$; in particular,
$C_1$ does not depend on derivatives of $\delta$.

On $D_0$ extend $\lambda\delta$ by zero across
$D_0\setminus D_1$ and put $a=\lambda\delta-u$.
On $D_1$ extend $(\lambda-1)\delta$ by zero across
$D_1\setminus D_0$ and put $b=(\lambda-1)\delta-u$.
The extended functions $\lambda\delta$ on $D_0$ and
$(\lambda-1)\delta$ on $D_1$ are continuous and have distributional
$\bar\partial$-derivative $v$: their cutoff factors are constant on
neighborhoods of the portions where the overlap ends inside the union.
Since $\bar\partial u=v$, it follows that
$\bar\partial a=\bar\partial b=0$ in the sense of distributions.  Weyl's
lemma therefore shows that $a$ and $b$ are holomorphic.  On the
overlap, $a-b=\delta=g-h$.  Consequently the two definitions
\[
 H=h+a-a(0)\quad\text{on }D_0,\qquad
 H=g+b-a(0)\quad\text{on }D_1
\]
agree and give a holomorphic map on $\D$.  It extends
continuously to $\overline\D$, and $H(0)=h(0)$.
The bound for $u$, together with $|\lambda|\le1$ and
$|a(0)|\le C_1\epsilon$, proves
\eqref{eq:additive-Cartan-estimate} with a constant depending
only on the pair and the chosen cutoff.  The final assertion
follows by comparing this bound with the stated positive
distances to the boundary of $\Omega$.
\end{proof}

We can now combine the target-side insertions with the source-side construction.
\begin{lemma}\label{lem:arc}
Let $h\in\cA(\Omega,z_0)$ with $z_0\notin A$, and let
$\zeta_1,\ldots,\zeta_m\in\T$ be distinct.  Set
$x_k=h(\zeta_k)$.  Assume that for every $k$ there are
$V_k\Subset\Omega$ containing $x_k$ and a holomorphic map
$q_k:V_k\times\overline\D\to\Omega$, defined on a neighborhood
of its domain, such that $q_k(y,0)=y$ and
\[
 Q_k(0)=f(x_k)\ne0,\qquad
 Q_k(\xi):=f(q_k(x_k,\xi))\not\equiv0.
\]
Let $b_{k,1},\ldots,b_{k,m_k}\in\D^*$ be all zeros of $Q_k$ in
$\D$, counted with multiplicities.  For every $\tau>0$ there are boundary
neighborhoods $I_k$ of $\zeta_k$, depending on $h$, the local families
$q_k$, and $\tau$, but independent of the later error tolerance, with the following
property.  For any pairwise disjoint closed arcs $J_k\Subset I_k$ of positive length, containing
$\zeta_k$ and having positive separation, for every $\epsilon>0$, and for
$r<1$ greater than the moduli of all zeros of $f\circ h$ in $\D$, there is
$H\in\cA(\Omega,z_0)$, arbitrarily close to $h$ on
$r\overline\D$, such that
\begin{equation}\label{eq:arc-estimate}
 H_A(H)\le H_A(h)
 +\sum_{k=1}^m |J_k|\sum_{\nu=1}^{m_k}\log|b_{k,\nu}|
 +\tau+\epsilon.
\end{equation}
\end{lemma}

\begin{proof}
The source-side construction follows the proof of
\cite[Theorem~3.1, especially (3.9)--(3.11)]{DDF12}.  The additional point
is that the target intersections with the divisor must be recovered by a
Rouch\'e argument.  The order in which the parameters are chosen is kept
explicit.

\smallskip
\noindent\emph{Step 1: The target zeros are separated first, and then local Rouch\'e margins are fixed.}
Let $\beta_{k,s}$ be the distinct zeros of $Q_k$ in $\D$, with
multiplicities $\mu_{k,s}$.  Around every $\beta_{k,s}$ choose pairwise disjoint open discs
$B_{k,s}\Subset\D^*$ so small that $Q_k$ has no zero on
$\partial B_{k,s}$ and no zero in $B_{k,s}$ other than
$\beta_{k,s}$.  We also require the logarithmic smallness condition in
Lemma~\ref{lem:source-preimage}; importantly, this extra restriction depends
only on $\beta_{k,s}$ and is therefore fixed before the source cutoff
$\chi_k$ is chosen.  Finally, the discs are chosen so small that
\begin{equation}\label{eq:source-cell-budget}
 \sum_s \mu_{k,s}\omega(B_{k,s})<\frac{\tau}{32}
 \qquad (k=1,\ldots,m).
\end{equation}
After shrinking a boundary neighborhood $I_k$ of $\zeta_k$, the image under
$h$ of a thin source neighborhood of $I_k$ lies compactly in $V_k$.
By shrinking $I_k$ once more, uniformly for $z$ in that source
neighborhood and $\xi\in\partial B_{k,s}$,
\[
 \bigl|f(q_k(h(z),\xi))-Q_k(\xi)\bigr|
 <\frac12\min_{\partial B_{k,s}}|Q_k|.
\]
Thus all Rouch\'e contours and their positive margins are fixed before the
final arcs $J_k\Subset I_k$ and before the large integer $N$ are chosen.

\smallskip
\noindent\emph{Step 2: The source pair is fixed before the holomorphic cutoffs.}
For the pairwise separated arcs $J_k$, choose disjoint attached source
discs $\Delta_k$ as in the proof of \cite[Theorem~3.1]{DDF12},
with $h(\overline\Delta_k)\subset V_k$.  Their closures can be
chosen mutually disjoint and in a collar disjoint from $r\overline\D$.
Choose boundary arcs
\[
 J_k\Subset \widehat J_k\Subset J_k^0\Subset J_k'
 \Subset \partial\Delta_k\cap\T.
\]
Now choose the dented central domain $D_0$, with
$r\overline\D\subset D_0$ and $\overline{D_0}\cap J_k'=\varnothing$
for every $k$, as in the cited construction.  Put
$D_1=\bigcup_k\Delta_k$.  The dent is placed so that
$D_0\cup D_1=\D$ and
\[
 \overline{D_0\setminus D_1}\cap
 \overline{D_1\setminus D_0}=\varnothing.
\]
In particular the domains, including the possibly disconnected $D_1$,
satisfy the separation required in Lemma~\ref{lem:cartan-additive}.
The cutoff $\lambda$ and constant $C$ in that lemma are determined
by this source pair, before $M$ and $N$ are chosen.
Before choosing the harmonic cutoffs, surround
the finitely many zeros of $f\circ h$ in $r\D$ by pairwise disjoint
source discs $O_\lambda\Subset D_0\setminus\overline{D_1}$.
Choose these discs so that the total multiplicity-weighted
oscillation of $\log|z|$ on them is less than $\tau/4$.  Set
\[
 c_{\rm old}:=\min_\lambda\min_{\partial O_\lambda}|f\circ h|>0,
 \qquad
 c_{\rm new}:=\frac12\min_{k,s}\min_{\partial B_{k,s}}|Q_k|>0.
\]
An empty family of old zeros or of prescribed new zeros
simply omits the corresponding margin and Rouch\'e check.
These margins depend on the fixed target discs and old-zero
neighborhoods, but not on the integer $N$.
On $\Delta_k$ put
\[
 a_k(z)=z^N\chi_k(z),\qquad \chi_k=e^{u_k+iv_k},
\]
where $u_k\le0$ is harmonic and smooth up to the boundary, equals $0$ on
$J_k^0$, equals a negative constant $-M$ on
$\partial\Delta_k\setminus J_k'$, and is smoothly interpolated on the two
short transition arcs.  Let $v_k$ be a harmonic conjugate of $u_k$.  The
harmonic-measure estimate in \cite[proof of Theorem~3.1]{DDF12} makes
$|a_k|$ as small as required on $D_0\cap\Delta_k$, uniformly in $N$.
On the fixed compact overlaps, the normalized harmonic cutoffs satisfy
$u_k\le-cM$ for some $c>0$ determined by the source pair and arcs:
these overlaps stay away from $J_k'$, and the harmonic measure of
$\partial\Delta_k\setminus J_k'$ has a positive minimum there.
Put $g_N(z)=q_k(h(z),a_k(z))$ on $\Delta_k$.  The compact sets
\[
 K_0=h(\overline{D_0}),\qquad
 K_1=\bigcup_k q_k(h(\overline{\Delta_k}),\overline\D)
\]
lie in $\Omega$ and do not depend on $M$ or $N$.
Let $d>0$ be smaller than their distances from $\partial\Omega$.
On fixed neighborhoods of these compact sets $f$ has a finite
local Lipschitz constant $L$ for pairs of sufficiently close points.
Since $q_k(y,0)=y$, compactness and
$|a_k|\le e^{-cM}$ on the overlaps give a constant $C_q$,
independent of $M$ and $N$, such that
\[
 \|g_N-h\|_{\overline{D_0\cap D_1}}\le C_q e^{-cM}.
\]
Choose $M$ so large that the bound
$C C_q e^{-cM}$ from Lemma~\ref{lem:cartan-additive}
is smaller than $d/2$ and the allowed local Lipschitz radius,
and, after multiplication by $L$,
smaller than $c_{\rm old}/2$ and $c_{\rm new}/2$
whenever the respective families are nonempty.
The resulting target and Rouch\'e margins are independent of $N$.
Now that $M$ and $\chi_k$ are fixed, shrink only the collar in which
source roots will be selected, leaving $D_0$ and $D_1$ unchanged.
Since $|\chi_k|=1$ on $J_k^0$ and every $|\beta_{k,s}|<1$, this
collar may be chosen so that
\[
 0<|\beta_{k,s}/\chi_k(z)|<1
\]
there for every $s$.
Only then choose $N$ sufficiently large for the source-root
construction and all the conditions below.

\smallskip
\noindent\emph{Step 3: Root selection is made only after the Cartan pair and cutoff are fixed.}
For each pair $(k,s)$ apply Lemma~\ref{lem:root-selection} on $\Delta_k$,
with $J=J_k$, $\chi=\chi_k$, and $\beta=\beta_{k,s}$.  The collar was chosen
in Step~2 so that $0<|\beta_{k,s}/\chi_k|<1$, and it may now be made thinner,
without changing $D_0$, $D_1$, the additive cutoff, or the gluing constant.
Choose positive numbers $\epsilon_{k,s}$ such that
\[
 \sum_{k,s}\mu_{k,s}\epsilon_{k,s}<\frac{\epsilon}{2}.
\]
For all sufficiently large $N$ we obtain selected solutions
$z_{k,s,\ell}$ of $a_k(z)=\beta_{k,s}$ satisfying
\begin{equation}\label{eq:selected-root-count}
 \#\{z_{k,s,\ell}\}\le 2N|J_k|
\end{equation}
and
\begin{equation}\label{eq:selected-root-weight}
 \sum_\ell\log|z_{k,s,\ell}|
 \le |J_k|\log|\beta_{k,s}|+\epsilon_{k,s}.
\end{equation}
This is exactly the point at which the large root-distribution integer is
chosen.  In particular, every target Rouch\'e margin and every constant in
the additive Cartan splitting was fixed before $N$.

\smallskip
\noindent\emph{Step 4: Rouch\'e's theorem transfers the source roots to divisor intersections.}
For every selected root $z_{k,s,\ell}$, apply
Lemma~\ref{lem:source-preimage} with its open boundary arc $J'$ chosen inside
the relative interior of $\widehat J_k$ and still containing $J_k$.  This is
legitimate because $u_k=0$ on $J_k^0$, hence
$|\chi_k|=1$ on a neighborhood of $\widehat J_k$ in
$\partial\Delta_k\cap\T$.  We obtain a simply connected source disc
$C_{k,s,\ell}\Subset\Delta_k$ on which
\[
 a_k:C_{k,s,\ell}\longrightarrow B_{k,s}
\]
is biholomorphic.  We take $C_{k,s,\ell}$ to be the corresponding connected component of
$a_k^{-1}(B_{k,s})$.  Hence the inverse branches corresponding to distinct
selected roots are pairwise disjoint; for distinct $s$ this also follows
from the disjointness of the target discs $B_{k,s}$.  The selected roots satisfy $|\log|z||=O(N^{-1})$ uniformly, and their
boundary traces lie compactly inside $\widehat J_k$.  Hence the uniform
form of Lemma~\ref{lem:source-preimage} applies simultaneously to all of
them.  By taking $N$ still larger, all these source discs lie in the thin
source neighborhood on which the Rouch\'e margin from Step~1 is valid.  Since $Q_k$ has the single zero
$\beta_{k,s}$ of multiplicity $\mu_{k,s}$ in $B_{k,s}$, the function
$Q_k\circ a_k$ has exactly $\mu_{k,s}$ zeros in
$C_{k,s,\ell}$, counted with multiplicity.  Rouch\'e's theorem therefore
gives exactly the same zero count for
\[
 z\longmapsto f(q_k(h(z),a_k(z))).
\]
No simplicity assumption on the divisor intersection is used here.

Let $w_{k,s,\ell,1},\ldots,
w_{k,s,\ell,\mu_{k,s}}$ denote these new zeros, counted with
multiplicity.  By \eqref{eq:source-cell-osc},
\[
 \sum_{p=1}^{\mu_{k,s}}\log|w_{k,s,\ell,p}|
 \le \mu_{k,s}\log|z_{k,s,\ell}|
 +\frac{\mu_{k,s}}{N}
  \bigl(\omega(B_{k,s})+o(1)\bigr).
\]
By \eqref{eq:selected-root-count}, there are at most $2N|J_k|$ selected roots
for each pair $(k,s)$.  Hence the source-cell oscillation contributes at most
\[
 2|J_k|\sum_s\mu_{k,s}\omega(B_{k,s})+o(1)
 <\frac{\tau|J_k|}{16}+o(1)
\]
by \eqref{eq:source-cell-budget}.  The contours and the neighborhoods $I_k$
were chosen using the same tolerance $\tau/32$ in every local model,
independently of the number and lengths of the later arcs.  After increasing $N$ once more, summing
first in $\ell$ and then in $s$, and using
\eqref{eq:selected-root-weight}, we obtain
\begin{equation}\label{eq:pre-glue-weight}
 \sum_{\mathrm{new\ zeros\ over}\ J_k}\log|z|
 \le |J_k|\sum_{\nu=1}^{m_k}\log|b_{k,\nu}|
      +\frac{\tau|J_k|}{8}+\epsilon_k,
\end{equation}
where $\epsilon_k:=\sum_s\mu_{k,s}\epsilon_{k,s}$ and
$\sum_k\epsilon_k<\epsilon/2$.

\smallskip
\noindent\emph{Step 5: Additive gluing and radial regularization preserve the retained weights.}
Apply Lemma~\ref{lem:cartan-additive} directly to the central map $h$
on $D_0$ and the map $g_N=q_k(h,a_k)$ on each component $\Delta_k$
of $D_1$.  The overlap estimate from Step~2 yields a disc $H^0$
holomorphic in $\D$, continuous on $\overline\D$, with $H^0(0)=z_0$
and $H^0(\overline\D)\subset\Omega$.  It is within
$C C_qe^{-cM}$ of $h$ on $\overline{D_0}$ and of $g_N$ on
$\overline{D_1}$, uniformly in $N$.  Since $r\overline\D\subset D_0$,
we can make it arbitrarily close to $h$ there.  By the tolerances
fixed before $N$, Rouch\'e's theorem preserves the old zero counts
in the $O_\lambda$.  On every $\partial C_{k,s,\ell}$,
Step~1 gives
\[
 |f(q_k(h(z),a_k(z)))|\ge c_{\rm new},
\]
since $a_k(\partial C_{k,s,\ell})=\partial B_{k,s}$.
Indeed, on this contour the first perturbation is less than
$\frac12\min_{\partial B_{k,s}}|Q_k|$, while the second is less than
$c_{\rm new}/2\le\frac14\min_{\partial B_{k,s}}|Q_k|$.  Hence
\[
 |f(H^0(z))-Q_k(a_k(z))|
 <\frac34\min_{\partial B_{k,s}}|Q_k|
 \le\frac34|Q_k(a_k(z))|.
\]
Rouch\'e therefore also gives exactly $\mu_{k,s}$ zeros of
$f\circ H^0$ in each $C_{k,s,\ell}$.  All these zeros stay
in the same source cell; hence \eqref{eq:source-cell-osc}
and \eqref{eq:selected-root-weight} give the bound
\eqref{eq:pre-glue-weight} directly for their logarithmic
weights.  No new smallness threshold is chosen after $N$.
The old zeros change their total logarithmic weight by less
than $\tau/4$ by the choice of the $O_\lambda$.

The class $\cA(\Omega)$ used in this paper consists of discs holomorphic on
a neighborhood of $\overline\D$.  To recover this class, choose
$\rho<1$ so close to $1$ that it is larger than the moduli of all the
finitely many old and prescribed new zeros just retained, and set
\[
 H(z):=H^0(\rho z).
\]
Then $H$ is holomorphic on $|z|<1/\rho$, maps a neighborhood of
$\overline\D$ into $\Omega$, satisfies $H(0)=z_0$, and remains
arbitrarily close to $h$ on $r\overline\D$.  Each retained zero $w$ of
$H^0$ becomes the zero $w/\rho$ of $H$.  If $N_{\rm ret}$ denotes the total
multiplicity of the finitely many retained zeros, their total logarithmic
weight increases by $-N_{\rm ret}\log\rho$.  Choosing $\rho$ still closer
to $1$ makes this increase smaller than $\tau/4$.  Any other
zeros contribute nonpositive logarithmic weights and may be discarded in an
upper estimate for $H_A(H)$.

Finally, summing \eqref{eq:pre-glue-weight} over $k$ gives an error less than
$\tau/2+\epsilon/2$.  The old-zero and radial-regularization
errors contribute less than the remaining $\tau/2$.
Together with the preserved old zeros this yields
\eqref{eq:arc-estimate}.
\end{proof}

\section{Consequences and the Hartogs hull viewpoint}\label{sec:consequences}

The bounded envelopes used in the proof also admit a Hartogs-hull
interpretation.  We record it because it connects the present
disc-functional approximation with the hull construction of \cite{Dieu05};
the source-side root distribution and gluing remain those of \cite{DDF12}.

\subsection{The Hartogs hull viewpoint}\label{sec:hartogs}

The construction in \cite[Lemmas 2.5--2.6]{Dieu05} begins with a
Hartogs compact set over $\partial\Omega$ of fiber radius $|f|$.
For compact $E\subset U\times\C$, with $U$ Stein, write
\[
 \widehat E^{\,U}:=
 \{(z,w)\in U\times\C:
   |F(z,w)|\le\sup_E|F|\text{ for every }
   F\in\mathcal O(U\times\C)\}.
\]
The ambient neighborhood matters: a biholomorphism defined near
$\overline D$ need not preserve the global polynomial hull.

For this subsection we write $\Psi$ for the convexifying biholomorphism
and reserve $\Phi$ for the Hartogs fiber function.

\begin{proposition}\label{prop:hartogs}
Assume that $D\Subset\C^n$ is convex and B-regular, that
$\Psi$ is biholomorphic from a neighborhood of $\overline D$ onto a
neighborhood of $\overline\Omega$, where $\Omega=\Psi(D)$, and retain the
normalization \eqref{eq:norm}.  Choose a convex neighborhood
$V\supset\overline D$ on which $\Psi$ is biholomorphic and such that
$f\circ\Psi$ is holomorphic, and put $U:=\Psi(V)$.  Let
\[
 K:=\{(p,w):p\in\partial\Omega,\ |w|\le |f(p)|\},
 \qquad
 r_K(z):=\sup\{|w|:(z,w)\in\widehat K^{\,U}\}.
\]
For $t>0$ set
\[
 \varphi_t:=P_\Omega\min\{-\log|f|,t\},
 \qquad
 \Phi(z):=\sup_{j\ge1}\varphi_j(z).
\]
Then
\begin{equation}\label{eq:hartogs-fiber-function}
 -\log r_K=\Phi\qquad\text{on }\Omega.
\end{equation}
In particular, $\Phi$ is lower semicontinuous.  If $A=(f)|_\Omega$ is
reduced, then
\begin{equation}\label{eq:hartogs-green-merged}
 \Phi=\widetilde G_A\quad\text{on }\Omega\setminus A,
 \qquad
 \Phi^*=\widetilde G_A\quad\text{on }\Omega.
\end{equation}
Equivalently,
\begin{equation}\label{eq:hartogs-radius-regularized}
 e^{-\widetilde G_A(z)}
 =\liminf_{\zeta\to z}r_K(\zeta),\qquad z\in\Omega,
\end{equation}
and $r_K=e^{-\widetilde G_A}$ pointwise on $\Omega\setminus A$.
\end{proposition}

\begin{proof}
By Proposition~\ref{thm:geometry}, $\Omega$ is B-regular and, when $A$ is reduced, the pair $(\Omega,A)$ admits an SID.  Hence
Proposition~\ref{prop:PB-trunc} gives
\[
 \varphi_t=PB_{\phi_t}\in\PSH(\Omega)\cap C(\overline\Omega),
 \qquad \varphi_t=\phi_t\quad\text{on }\partial\Omega.
\]
Fix $t>0$ and introduce, only for this proof, the truncated Hartogs compact
set
\[
 K_t:=\{(p,w):p\in\partial\Omega,
       \ |w|\le\max\{|f(p)|,e^{-t}\}\}.
\]
We claim that its holomorphic hull in $U\times\C$ satisfies
\begin{equation}\label{eq:hartogs-bounded}
 \widehat K_t^{\,U}\cap(\{z\}\times\C)
 =\{(z,w):|w|\le e^{-\varphi_t(z)}\},
 \qquad z\in\Omega.
\end{equation}

The closure of $D$ is convex and therefore holomorphically convex in $V$;
the same is true of $\overline\Omega$ in $U$.  Thus the base projection of
$\widehat K_t^{\,U}$ lies in $\overline\Omega$, and its $w$-coordinate has
modulus at most $1$ by \eqref{eq:norm}.

First take $z\in\Omega$ and $|w|<e^{-\varphi_t(z)}$.  Let
$F\in\mathcal O(U\times\C)$ satisfy $\sup_{K_t}|F|\le1$, and write
$F(z,w)=\sum_{k\ge0}a_k(z)w^k$.  The Cauchy estimate on the vertical disc
over $p\in\partial\Omega$, whose radius is $e^{-\phi_t(p)}$, gives
\[
 |a_k(p)|\le e^{k\phi_t(p)}\quad(k\ge1),
 \qquad |a_0(p)|\le1.
\]
The maximum principle yields $|a_0|\le1$ in $\Omega$.  For $k\ge1$,
$(1/k)\log|a_k|$ is plurisubharmonic and has boundary upper limit at most
$\phi_t$; Perron--Bremermann comparison gives
$|a_k(z)|\le e^{k\varphi_t(z)}$.  Hence
\[
 |F(z,w)|\le\frac{1}{1-e^{\varphi_t(z)}|w|}.
\]
Applying the same estimate to $F^m$ for every positive integer $m$ and
letting $m\to\infty$ gives $|F(z,w)|\le1$.  Thus
$(z,w)\in\widehat K_t^{\,U}$, and the boundary circle follows by closedness.

For the reverse inclusion fix $w_0\in D$ and, for small $s>0$, set
\[
 V_s:=\{\xi\in V:(1-s)\xi+sw_0\in D\},
 \qquad U_s:=\Psi(V_s).
\]
The domain $V_s$ is convex, contains $\overline D$, and is Runge in $V$;
therefore $U_s\times\C$ is Runge in $U\times\C$.  Define
\[
 \beta_s(z):=\varphi_t\bigl(
 \Psi((1-s)\Psi^{-1}(z)+sw_0)\bigr),\qquad z\in U_s,
\]
and
\[
 \delta_s:=\max\Bigl\{0,
 \max_{p\in\partial\Omega}(\beta_s(p)-\phi_t(p))\Bigr\}.
\]
Then $\beta_s\to\varphi_t$ uniformly on $\overline\Omega$ and
$\delta_s\to0$.  The plurisubharmonic function
$\log|w|+\beta_s(z)-\delta_s$ is at most zero on $K_t$.  Since holomorphic
and plurisubharmonic hulls agree for compact sets in a Stein manifold, it is
at most zero on the holomorphic hull of $K_t$ computed in
$U_s\times\C$.  By the standard invariance of holomorphic hulls under a Runge inclusion,
a point of $\widehat K_t^{\,U}$ which lies in $U_s\times\C$ also belongs to
this latter hull.  Indeed, if a function
$F\in\mathcal O(U_s\times\C)$ separated that point from $K_t$, the Runge
density of $\mathcal O(U\times\C)$ in $\mathcal O(U_s\times\C)$ would
approximate $F$ uniformly on the compact set consisting of $K_t$ and the
single point, preserving the strict separating inequality.  This would
contradict membership in $\widehat K_t^{\,U}$.  Consequently every
$(z,w)\in\widehat K_t^{\,U}$ with $z\in\Omega$ satisfies
\[
 \log|w|+\beta_s(z)\le\delta_s.
\]
Letting $s\downarrow0$ proves the reverse inclusion in
\eqref{eq:hartogs-bounded}.

As $j\to\infty$, the compact sets $K_j$ decrease to $K$.  Their holomorphic
hulls also decrease to $\widehat K^{\,U}$: if a point is outside
$\widehat K^{\,U}$, a separating function in $\mathcal O(U\times\C)$
remains separating for $K_j$ for all sufficiently large $j$.  Intersecting
the fibers in \eqref{eq:hartogs-bounded} gives
\[
 \widehat K^{\,U}\cap(\{z\}\times\C)
 =\{(z,w):|w|\le e^{-\Phi(z)}\}.
\]
This proves \eqref{eq:hartogs-fiber-function}; lower semicontinuity of
$\Phi$ follows also from $\Phi=\sup_j\varphi_j$.

If $A$ is reduced, Proposition~\ref{thm:geometry} supplies the SID required by Theorem~\ref{thm:main}, and that theorem gives
$\Phi=\widetilde G_A$ on $\Omega\setminus A$ and
$\Phi^*=\widetilde G_A$ on $\Omega$.  Since $r_K=e^{-\Phi}$, this is
equivalent to \eqref{eq:hartogs-radius-regularized}.
\end{proof}

Thus $\Phi$ is the lower semicontinuous fiber function of
\cite[Lemma~2.6]{Dieu05}, expressed as the increasing limit of the bounded
envelopes $\varphi_j$.  On the B-regular domain $\Omega$,
Proposition~\ref{prop:PB-trunc} gives $\varphi_j=PB_{\phi_j}$, so
\[
 \Phi=\sup_{j\ge1}PB_{\phi_j}.
\]
Proposition~\ref{prop:hartogs} is therefore the Hartogs-hull form of the
regularization relation in Remark~\ref{rem:PB-defect}.  In
\cite[Theorem~2.2]{Dieu05}, homogeneity of $f$ and a radial inward condition
give the comparison with the Green function.  Here the SID approximation and disc insertion give $\Phi=\widetilde G_A$ off $A$; equality
of the upper regularizations does not assert pointwise equality on $A$.

\subsection{A fiber formula}

The preceding hull discussion concerns the bounded truncations associated
with a principal divisor.  The next elementary fiber formula is inspired by
the Hartogs example in \cite[Proposition~3.1]{Dieu05}; we record it in a
slightly more flexible form allowing a holomorphic center and an arbitrary
positive multiplicity.

\begin{proposition}\label{prop:Hartogs-fibers}
Let $D\subset\C^{n-1}$ be a domain, let
$\psi\in\PSH(D)\cap C(D)$ be real valued, and let $a\in\mathcal O(D)$.
Suppose that
\[
 \Omega_{\psi,a}
 =\{(z,w)\in D\times\C:|w-a(z)|<e^{-\psi(z)}\}
\]
is bounded.  For any integer $m\ge1$ put
$A=m[\{w=a(z)\}\cap\Omega_{\psi,a}]$.  Then
\begin{equation}\label{eq:Hartogs-fiber-Green}
 G_{A,\Omega_{\psi,a}}(z,w)
 =m\bigl(\log|w-a(z)|+\psi(z)\bigr),
\end{equation}
with the value $-\infty$ on $A$.  Consequently
$G_{A,\Omega_{\psi,a}}\in C(\Omega_{\psi,a},[-\infty,0])$.
\end{proposition}

\begin{proof}
The holomorphic change of variable $(z,w)\mapsto(z,w-a(z))$ reduces the
proof to $a=0$.  Write $R(z)=e^{-\psi(z)}$.  The function
\[
 q(z,w)=m\bigl(\log|w|+\psi(z)\bigr)
\]
is plurisubharmonic, nonpositive on $\Omega_{\psi,0}$, and has divisorial
multiplicity $m$ along $w=0$.  It is therefore a competitor in the definition
of $G_{A,\Omega_{\psi,0}}$, proving $q\le G_A$.

Conversely, let $u$ be any competitor and fix $z\in D$.  If
$u(z,\cdot)\equiv-\infty$, there is nothing to prove on that fiber.
Otherwise $u(z,\cdot)$ is a negative subharmonic function on
$|w|<R(z)$ with a pole of order at least $m$ at $w=0$.  The Lelong-number
condition implies that $u(z,w)-m\log|w|$ is bounded above near $w=0$ and
therefore has a removable subharmonic singularity there.  On $|w|=r<R(z)$
this function is at most $-m\log r$.  The one-variable maximum principle,
followed by $r\uparrow R(z)$, gives
\[
 u(z,w)\le m\log\frac{|w|}{R(z)}
 =m\bigl(\log|w|+\psi(z)\bigr).
\]
Taking the supremum over all competitors proves
\eqref{eq:Hartogs-fiber-Green}.  The continuity assertion follows directly
from the displayed formula.
\end{proof}

\section{Sectorial inward deformations}\label{sec:sectorial-proof}

We prove Theorem~\ref{thm:main}.  The argument has two target-side models near the boundary.  At a regular point of the divisor, a submersion coordinate gives a degree-one transverse insertion.  At a singular boundary point, the deformation parameter itself is used as a Weierstrass variable.  The source-side gluing is still Lemma~\ref{lem:arc}.  Singular points of $A$ which stay in a compact subset of $\Omega$ are handled by the ordinary interior construction below.

\begin{lemma}\label{lem:compact-interior-insertion}
Let $K\Subset\Omega$ be compact.  There are finitely many open sets $V_j$
covering $A\cap K$, vectors $e_j\in\C^n$, integers $m_j\ge1$, and
constants $a_0>0$ and $L>1$ with the following property.  If
$y\in V_j$ is sufficiently close to $A$, $f(y)\ne0$, and $0<a<a_0$, then,
for every choice of
\[
 r\in[L a^{1/m_j},2L a^{1/m_j}]
\]
which avoids the moduli of the relevant Weierstrass roots, there is a
holomorphic family $q:V_y\times\D\to\Omega$, defined on a neighborhood of its
domain, with $q(z,0)=z$, such that the zeros $b_1,\ldots,b_N\in\D^*$ of
$\eta\mapsto f(q(y,\eta))$ satisfy
\begin{equation}\label{eq:compact-insertion}
 \sum_{\nu=1}^N\log|b_\nu|
 \le \log|f(y)|-\log a.
\end{equation}
The constants can be chosen uniformly over the finite family of charts.
\end{lemma}

\begin{proof}
Fix $p\in A\cap K$ and choose a vector $e_p\in\C^n$ for which
$\tau\mapsto f(p+\tau e_p)$ is not identically zero.  Parametric Weierstrass
preparation, applied to $f(z+\tau e_p)$ near $(p,0)$, gives
\[
 f(z+\tau e_p)=U_p(z,\tau)W_p(z,\tau),
\]
where $U_p$ is nowhere zero and $W_p(z,\cdot)$ is monic of some finite degree
$m_p$.  Since $K\Subset\Omega$, after shrinking the parameter neighborhoods
there is a fixed Euclidean margin to the boundary, so the discs
$z+r\eta e_p$ remain in $\Omega$ whenever $r$ is sufficiently small.

Take a finite subcover and uniform upper and lower bounds for the units.  If
$\alpha_1(y),\ldots,\alpha_m(y)$ are the roots of the Weierstrass polynomial
in one chart and $r$ avoids their moduli, define
$q(z,\eta)=z+r\eta e_j$.  The zeros in $\D$ are
$b_\nu=\alpha_\nu(y)/r$ for $|\alpha_\nu(y)|<r$.  Hence
\[
 \sum_{|\alpha_\nu|<r}\log\frac{|\alpha_\nu|}{r}
 =\log|f(y)|-
 \log\!\left(|U_j(y,0)|\prod_{\nu=1}^{m_j}
              \max\{|\alpha_\nu(y)|,r\}\right).
\]
Choose $L$ so large that the last factor is at least $a$ whenever
$r\ge L a^{1/m_j}$.  This gives \eqref{eq:compact-insertion}.
\end{proof}

Near a regular boundary point the divisor itself supplies a transverse coordinate.
\begin{lemma}\label{lem:regular-boundary-insertion}
Let $K$ be a compact subset of $\operatorname{Reg}(A)\cap\partial\Omega$.  There are a neighborhood $U_0$ of $K$ and constants $\delta,\rho>0$ with the following property.  If $0<t<t_0$, $y\in U_0\cap\Omega$,
\[
 \dist(y,\partial\Omega)\ge ct,
 \qquad 0<|f(y)|<\delta t,
\]
then there are a neighborhood $V_y\Subset\Omega$ of $y$ and a holomorphic map $q:V_y\times\D\to\Omega$, defined on a neighborhood of its domain, with $q(z,0)=z$, for which $\eta\mapsto f(q(y,\eta))$ has one zero $b\in\D^*$ and
\begin{equation}\label{eq:regular-boundary-weight}
 \log|b|\le \log|f(y)|-\log t-\log\rho.
\end{equation}
Here $c$ is the constant in \eqref{eq:sector-depth}, after decreasing the boundary neighborhood if necessary.
\end{lemma}

\begin{proof}
The holomorphic submersion theorem gives finitely many coordinate neighborhoods of $K$ with coordinates
\[
 \chi_\alpha=(\chi_\alpha',w_\alpha),\qquad w_\alpha=f.
\]
After shrinking, the first derivatives of these maps and their inverses are uniformly bounded.  Choose $\rho>0$ so small that changing the $w_\alpha$-coordinate by at most $2\rho t$ moves a point by less than $ct/4$.  For $y$ in a smaller chart put
\[
 q(y,\eta):=\chi_\alpha^{-1}
 \bigl(\chi_\alpha'(y),f(y)+\rho t\eta\bigr).
\]
The same formula defines a holomorphic family for $z$ near $y$, and its image is contained in $B(y,ct/2)\subset\Omega$.  If $\delta<\rho/2$, the unique zero is
$b=-f(y)/(\rho t)\in\D^*$, and \eqref{eq:regular-boundary-weight} is an equality.
\end{proof}

We next record the singular boundary model.  The point of the finite-order escape condition is precisely to make the deformation parameter a Weierstrass variable.
\begin{lemma}\label{lem:sectorial-WPT}
Let
\[
 \Sigma:=\operatorname{Sing}(A)\cap\partial\Omega.
\]
If $\Sigma\ne\varnothing$, there are finitely many open sets $V_\alpha$ covering $\Sigma$, numbers $\tau_0,u_-,u_+>0$, and integers $m_\alpha\ge1$ such that
\begin{equation}\label{eq:sectorial-WPT}
 f(S_\tau(x))=U_\alpha(x,\tau)
 \prod_{\nu=1}^{m_\alpha}(\tau-\alpha_{\alpha,\nu}(x)),
 \qquad x\in V_\alpha,\quad |\tau|<\tau_0,
\end{equation}
where the roots are understood as an unordered multiset and
\[
 u_-\le |U_\alpha(x,\tau)|\le u_+.
\]
After shrinking the sets and $\tau_0$, every $S_\tau$ is locally biholomorphic on a neighborhood of $\overline\Omega$ for $|\tau|<\tau_0$.
\end{lemma}

\begin{proof}
For $p\in\Sigma$ the germ $\tau\mapsto f(S_\tau(p))$ is not identically zero by \eqref{eq:finite-escape}; let $m(p)$ be its order at zero.  Parametric Weierstrass preparation applied to $(x,\tau)\mapsto f(S_\tau(x))$ at $(p,0)$ gives \eqref{eq:sectorial-WPT} with degree $m(p)$.  Compactness of $\Sigma$ gives a finite subcover and uniform bounds for the units after shrinking.  Finally, $D_zS_0=I$ uniformly near the compact set $\overline\Omega$; hence $\det D_zS_\tau$ stays nonzero there for all sufficiently small $|\tau|$, and the holomorphic inverse function theorem gives the last assertion.
\end{proof}

The preceding factorization yields the sectorial Weierstrass insertion estimate needed below.
\begin{lemma}\label{lem:sectorial-insertion}
Fix $0<\theta'<\theta$.  Let $\tau_0$ be the parameter radius in Lemma~\ref{lem:sectorial-WPT}, and choose
\[
 0<\tau_1<\frac12\min\{\tau_0,t_0\}.
\]
Then the factorization in \eqref{eq:sectorial-WPT} is valid throughout $|\sigma|<2\tau_1$, while the deformation is defined on the larger disc $|\sigma|<t_0$.  There are $0<\kappa<1/4$, $L>1$, and $C_0>0$ with the following properties.  Let $x\in V_\alpha\cap\Omega$, let
\[
 \tau\in\Gamma_{\theta',\tau_1},\qquad t:=|\tau|,
 \qquad y:=S_\tau(x),
\]
and assume $f(y)\ne0$.
\begin{enumerate}
\item If $0<s<1$ and
\[
 2Ls^{1/m_\alpha}\le\kappa t,
\]
then one can choose
$r\in[Ls^{1/m_\alpha},2Ls^{1/m_\alpha}]$ avoiding the finitely many numbers $|\tau-\alpha_{\alpha,\nu}(x)|$ and find a neighborhood $V_y$ of $y$ and a holomorphic family $q:V_y\times\D\to\Omega$, $q(z,0)=z$, such that the zeros $b_\nu\in\D^*$ of $f(q(y,\cdot))$ satisfy
\begin{equation}\label{eq:sectorial-absolute}
 \sum_\nu\log|b_\nu|\le \log|f(y)|-\log s.
\end{equation}
\item If in addition $f(x)\ne0$, one can choose
$r\in[\kappa t/2,\kappa t]$ avoiding the same root distances and obtain
\begin{equation}\label{eq:sectorial-relative}
 \sum_\nu\log|b_\nu|
 \le \log|f(y)|-\log|f(x)|+C_0.
\end{equation}
\end{enumerate}
The constants are uniform over the finite family of singular boundary charts.
\end{lemma}

\begin{proof}
Choose $\kappa>0$ so small that
\[
 \tau+r\overline\D\subset\Gamma_{\theta,t_0}
 \quad\hbox{and}\quad
 |\tau+r\eta|<2\tau_1\quad(|\eta|\le1)
\]
whenever $\tau\in\Gamma_{\theta',\tau_1}$ and $0<r\le\kappa|\tau|$.  The second condition keeps the whole parameter disc inside the Weierstrass neighborhood fixed above.  Let $R_\tau$ be the local inverse of $S_\tau$ near $y$ and define
\begin{equation}\label{eq:sectorial-family}
 q(z,\eta):=S_{\tau+r\eta}(R_\tau(z)).
\end{equation}
After shrinking $V_y$, one has $R_\tau(V_y)\Subset\Omega$.  Since every parameter $\tau+r\eta$ remains in the sector of Definition~\ref{def:sectorial}, \eqref{eq:sector-selfmap} gives $q(V_y\times\D)\subset\Omega$.  The containment of the parameter disc in the larger sector is strict, so after one further shrinking the same formula is defined on a neighborhood of $V_y\times\overline\D$, as required by the disc-gluing lemma.

Write $m=m_\alpha$, $U=U_\alpha$, and $d_\nu:=|\tau-\alpha_{\alpha,\nu}(x)|$.  At the reference point $y$ the zeros of $\eta\mapsto f(q(y,\eta))$ in $\D$ are exactly
\[
 b_\nu=\frac{\alpha_{\alpha,\nu}(x)-\tau}{r},
 \qquad d_\nu<r,
\]
with multiplicity.  Hence
\begin{equation}\label{eq:sectorial-product}
 \sum_{d_\nu<r}\log|b_\nu|
 =\log|f(y)|-
 \log\!\left(|U(x,\tau)|\prod_{\nu=1}^m\max\{d_\nu,r\}\right).
\end{equation}
Choose $L$ so large that $u_-L^m\ge1$ for every degree in the finite cover.  In part (i), $r\ge Ls^{1/m}$ makes the factor inside the last logarithm at least $s$, proving \eqref{eq:sectorial-absolute}.

For part (ii), compare the Weierstrass factorizations at $0$ and at $\tau$.  If $I=\{\nu:d_\nu<r\}$, then
\begin{align*}
 &\sum_{\nu\in I}\log\frac{d_\nu}{r}
 -\bigl(\log|f(y)|-\log|f(x)|\bigr)\\
 &\qquad=-\log\left|\frac{U(x,\tau)}{U(x,0)}\right|
 +\sum_{\nu\in I}\log\frac{|\alpha_{\alpha,\nu}(x)|}{r}
 +\sum_{\nu\notin I}\log
 \frac{|\alpha_{\alpha,\nu}(x)|}{|\tau-\alpha_{\alpha,\nu}(x)|}.
\end{align*}
For every $\nu$, the inequalities $r\ge\kappa t/2$ and
$|\alpha_{\alpha,\nu}(x)|\le t+d_\nu$ give
\[
 |\alpha_{\alpha,\nu}(x)|
 \le\left(1+\frac{2}{\kappa}\right)\max\{d_\nu,r\}.
\]
Thus both root ratios in the preceding sum are at most $1+2/\kappa$.
Writing $m_*:=\max_\alpha m_\alpha$, one may take
\[
 C_0:=\log\frac{u_+}{u_-}
       +m_*\log\left(1+\frac{2}{\kappa}\right).
\]
This constant is independent of $x$, $\tau$, and the later approximation level.  Together with \eqref{eq:sectorial-product} this gives \eqref{eq:sectorial-relative}.
\end{proof}

The next elementary moving-center lemma replaces any fixed-center requirement.
\begin{lemma}\label{lem:moving-center}
Let $u$ be plurisubharmonic near $p$, with $u(p)>-\infty$, and let
$q:\Delta_{r_0}\to\C^n$ be holomorphic with $q(0)=p$ and image in that neighborhood.  If $I\subset\T$ is a fixed closed arc of normalized length $\alpha>0$, then for every sequence $t_j\downarrow0$ there are $\xi_j\in I$ such that
\begin{equation}\label{eq:moving-center}
 u(q(t_j\xi_j))\ge u(p)-o(1).
\end{equation}
\end{lemma}

\begin{proof}
Put $U(\tau)=u(q(\tau))$ and
$M(t)=\sup_{|\tau|\le t}U(\tau)$.  Upper semicontinuity and $q(0)=p$ give $M(t)\downarrow U(0)$.  The submean inequality yields
\[
 U(0)\le\int_\T U(t\xi)\,d\sigma(\xi)
 \le \alpha\sup_{\xi\in I}U(t\xi)+(1-\alpha)M(t).
\]
Thus
\[
 \sup_{\xi\in I}U(t\xi)
 \ge U(0)-\frac{1-\alpha}{\alpha}\bigl(M(t)-U(0)\bigr),
\]
and the supremum on the compact arc is attained by upper semicontinuity.
\end{proof}

\begin{proof}[Proof of Theorem~\ref{thm:main}]
Fix $p\in\Omega\setminus A$ and retain the normalization \eqref{eq:norm}.  Write
\[
 \phi=-\log|f|,\qquad \phi_j=\min\{\phi,j\},\qquad s_j=e^{-j}.
\]
Choose a smaller angle $0<\theta'<\theta$.  If the singular boundary set is nonempty, fix the constants $\tau_1,\kappa,L,C_0$ from Lemma~\ref{lem:sectorial-insertion}, let
\[
 m_*:=\max_\alpha m_\alpha,
 \qquad t_{\rm sec}:=\tau_1.
\]
If the singular boundary set is empty, put $t_{\rm sec}:=t_0$; the constants from Lemma~\ref{lem:sectorial-insertion} are then not used.

Let $\Sigma=\operatorname{Sing}(A)\cap\partial\Omega$.  For each singular
boundary chart from Lemma~\ref{lem:sectorial-WPT}, choose a smaller open core
$V_\alpha'\Subset V_\alpha$ so that the cores cover $\Sigma$.  Next choose a
small neighborhood $U_{\rm sing}$ of $\Sigma$ whose closure is contained in
the union of these cores; if $\Sigma=\varnothing$, put $U_{\rm sing}=\varnothing$.  The compact set
\[
 K_{\rm reg}:=(A\cap\partial\Omega)\setminus U_{\rm sing}
\]
is contained in $\operatorname{Reg}(A)$.  Cover it by finitely many smaller
regular cores $U_\beta'\Subset U_\beta$ to which
Lemma~\ref{lem:regular-boundary-insertion} applies.  All boundary parent
charts are chosen inside the neighborhood $V$ from
Definition~\ref{def:sectorial}.

The remaining set
\[
 K_A:=(A\cap\overline\Omega)\setminus
       \bigl(U_{\rm sing}\cup\textstyle\bigcup_\beta U_\beta'\bigr)
\]
is compactly contained in $\Omega$.  Cover it by finitely many smaller
interior cores $W_\gamma'\Subset W_\gamma\Subset\Omega$ coming from
Lemma~\ref{lem:compact-interior-insertion}.  For each nonempty cover, fix the insertion constants uniformly over its finitely many charts; in particular, denote the regular insertion constants by $\delta,\rho$.  The union of all the smaller
cores is an open neighborhood of $A\cap\overline\Omega$.  Hence there is
$\gamma_0>0$ such that
\begin{equation}\label{eq:small-f-core-cover}
 |f(z)|<\gamma_0,\quad z\in\overline\Omega
 \quad\Longrightarrow\quad
 z\in\bigcup_\alpha V_\alpha'\cup
       \bigcup_\beta U_\beta'\cup
       \bigcup_\gamma W_\gamma'.
\end{equation}

Choose $A_0>1$ so large that the singular absolute insertion inequalities hold, namely
\begin{equation}\label{eq:sectorial-scale-choice}
 2Ls^{1/m}\le\kappa A_0s^{1/m_*}
 \qquad(0<s<1,\ 1\le m\le m_*),
\end{equation}
where this condition is omitted when the singular boundary set is empty.  If there are regular boundary charts, enlarge $A_0$ further so that
$A_0\delta\ge1$ and $A_0\rho\ge1$.  When the singular boundary set is empty, put $m_*=1$ after these choices.  Finally set
\[
 t_j:=A_0s_j^{1/m_*}.
\]

By the disc formula \eqref{eq:vt-disc}, choose $g_j\in\cA(\Omega,p)$ such that
\begin{equation}\label{eq:sectorial-nearmin}
 \int_\T\phi_j(g_j)\,d\sigma
 \le v_j(p)+j^{-1}
 \le \widetilde G_A(p)+1=:C_*.
\end{equation}
For the fixed point $p$, the map $\tau\mapsto S_\tau(p)$ stays in $\Omega$ on a full disc about $0$ when $|\tau|$ is sufficiently small.  Apply Lemma~\ref{lem:moving-center} to
$u=\widetilde G_A$, $q(\tau)=S_\tau(p)$, and the arc
\[
 I:=\{e^{i\vartheta}:|\vartheta|\le\theta'\}.
\]
After discarding finitely many indices, there are
$\tau_j=t_j\xi_j\in\Gamma_{\theta',t_{\rm sec}}$ such that, with
$p_j:=S_{\tau_j}(p)$,
\begin{equation}\label{eq:sectorial-center}
 \widetilde G_A(p_j)\ge\widetilde G_A(p)-o(1).
\end{equation}
Since $p\notin A$, also $p_j\notin A$ for large $j$.  Put
\[
 h_j:=S_{\tau_j}\circ g_j\in\cA(\Omega,p_j).
\]
By \eqref{eq:sector-displacement},
\[
 \|h_j-g_j\|_{\overline\D}\le Ct_j.
\]
Since $f$ is Lipschitz on a fixed neighborhood of $\overline\Omega$, there is $M>0$ such that
\begin{equation}\label{eq:sectorial-f-close}
 |f(h_j)-f(g_j)|\le Mt_j\qquad\text{on }\overline\D.
\end{equation}

\smallskip
\noindent\emph{Step 1: The exceptional set has length $O(1/j)$.}
Put
\[
 E_j:=\{\zeta\in\T:|f(g_j(\zeta))|<s_j\}.
\]
Since the integrand in \eqref{eq:sectorial-nearmin} equals $j$ on $E_j$,
\begin{equation}\label{eq:sectorial-E}
 \sigma(E_j)\le C_*/j.
\end{equation}
Let $\eta_j=j^{-1/2}$ and define
\[
 \Psi_j:=-\log|f(h_j)|-\phi_j(g_j),
 \qquad B_j:=\{\Psi_j>\eta_j\},
\]
with the value $+\infty$ at the finitely many boundary zeros of $f\circ h_j$.  On $B_j\setminus E_j$,
\[
 |f(h_j)|<e^{-\eta_j}|f(g_j)|,
\]
and \eqref{eq:sectorial-f-close} gives
\[
 (1-e^{-\eta_j})|f(g_j)|\le Mt_j.
\]
Hence
\begin{equation}\label{eq:sectorial-fgj}
 |f(g_j)|\le C_2\frac{t_j}{\eta_j}
 \qquad\text{on }B_j\setminus E_j.
\end{equation}
Since $t_j=A_0e^{-j/m_*}$,
\[
 \phi_j(g_j)
 \ge \frac{j}{m_*}-\frac12\log j-C_3
 \qquad\text{on }B_j\setminus E_j.
\]
Together with \eqref{eq:sectorial-nearmin} and \eqref{eq:sectorial-E}, this gives
\begin{equation}\label{eq:sectorial-B}
 \sigma(B_j)=O(j^{-1}).
\end{equation}
Moreover, $|f(g_j)|\to0$ and $|f(h_j)|\to0$ uniformly on $B_j$.

\smallskip
\noindent\emph{Step 2: The bad set is localized in the fixed core charts.}
Because the closures of the cores are compactly contained in their parent
charts, there is a positive minimum Euclidean margin from every core to the
complement of its parent.  Since $\|h_j-g_j\|_{\overline\D}\to0$, and since
$|f(g_j)|$ and $|f(h_j)|$ tend uniformly to zero on $B_j$, for all large
$j$ every $g_j(\zeta)$ with $\zeta\in B_j$ lies in a smaller core and the
corresponding $h_j(\zeta)$ lies in its parent chart.  We use the priority
singular boundary, regular boundary, then compact interior when cores overlap.
In each boundary parent chart, \eqref{eq:sector-depth} gives
\begin{equation}\label{eq:sectorial-local-depth}
 \dist(h_j(\zeta),\partial\Omega)\ge c|\tau_j|=ct_j.
\end{equation}

\smallskip
\noindent\emph{Step 3: Each region recovers the logarithmic loss.}
Work away from the finitely many boundary zeros of $f\circ h_j$; short arcs around them will be removed in Step~4.

If $g_j(\zeta)$ lies in a singular boundary chart, use Lemma~\ref{lem:sectorial-insertion}.  For $\zeta\in E_j\cap B_j$, part (i) with $s=s_j$ is allowed by \eqref{eq:sectorial-scale-choice} and gives an inserted weight $W_\zeta$ satisfying
\[
 W_\zeta\le\log|f(h_j(\zeta))|-\log s_j
 =-\Psi_j(\zeta).
\]
For $\zeta\in B_j\setminus E_j$, part (ii) gives
\begin{equation}\label{eq:sectorial-singular-relative}
 W_\zeta\le-\Psi_j(\zeta)+C_0.
\end{equation}

If the point lies in a regular boundary chart, apply Lemma~\ref{lem:regular-boundary-insertion} with $t=t_j$.  On $E_j\cap B_j$ one has
$|f(h_j)|<s_j$.  Since $s_j/t_j\to0$ when $m_*>1$, and in the case $m_*=1$ the constant $A_0$ was chosen sufficiently large, the hypothesis $|f(h_j)|<\delta t_j$ holds and
\[
 W_\zeta\le-\Psi_j(\zeta)
\]
for all large $j$.  On $B_j\setminus E_j$, whenever $|f(h_j)|<\delta t_j$, the same lemma and \eqref{eq:sectorial-fgj} give
\begin{equation}\label{eq:sectorial-regular-relative}
 W_\zeta\le-\Psi_j(\zeta)+C_4+\log(1/\eta_j).
\end{equation}

Finally, in the compact interior region $W$, apply
Lemma~\ref{lem:compact-interior-insertion} with $a=s_j$ on $E_j$ and
$a=|f(g_j)|$ off $E_j$.  The parameters tend uniformly to zero on the
relevant subsets of $B_j$, and the inserted weight again satisfies
$W_\zeta\le-\Psi_j(\zeta)$.

Let $Z_j\subset\T$ be the finite set of boundary zeros of $f\circ h_j$.
Let $G_j\subset B_j\setminus Z_j$ be the union of the relatively open source
sets on which at least one of the preceding insertion families is admissible
(singular boundary, regular boundary with $|f(h_j)|<\delta t_j$, or compact
interior).  Thus $G_j$ is relatively open in $B_j\setminus Z_j$.  Define
\[
 D_j:=(B_j\setminus Z_j)\setminus G_j.
\]
By the finite core cover from Step~2, every point of $D_j$ lies in the
regular boundary region, belongs to $B_j\setminus E_j$, and satisfies
$|f(h_j)|\ge\delta t_j$; otherwise one of the three insertion constructions
would apply.  Thus \eqref{eq:sectorial-fgj} gives
\[
 0<\Psi_j\le C_5+\log(1/\eta_j)=O(\log j)
 \qquad\text{on }D_j.
\]
Together with \eqref{eq:sectorial-B},
\begin{equation}\label{eq:sectorial-uninserted}
 \int_{D_j}\Psi_j\,d\sigma
 =O\!\left(\frac{\log j}{j}\right)=o(1).
\end{equation}

\smallskip
\noindent\emph{Step 4: Arc insertion globalizes the local corrections.}
Choose a sequence $\varepsilon_j\downarrow0$.  For each fixed $j$, first remove arbitrarily short arcs around the finitely many boundary zeros of $f\circ h_j$ so that the integral of $\Psi_j^+$ over the removed arcs is at most $\varepsilon_j$.  On the remaining part of $B_j\setminus D_j$, the function $\Psi_j$ is continuous.

For every point $\zeta$ in this remaining set, choose one of the
insertion families from Step~3 and a sufficiently small relatively open
neighborhood of $\zeta$ contained in $G_j$ on which this same family remains
admissible.  First fix all target zero contours for that family.  Apply the preparatory part of Lemma~\ref{lem:arc}, with the common
tolerance $\tau=\varepsilon_j$, to obtain a boundary neighborhood
$I_\zeta$ on which the target contours and their Rouch\'e margins remain
valid.  Closed arcs centered at $\zeta$, contained in
$I_\zeta\cap G_j$, and sufficiently short that the oscillation of $\Psi_j$
is at most $\varepsilon_j$, form a Vitali cover of the remaining set.  Thus
every selected arc is contained in $B_j\setminus D_j$.  The one-dimensional
Vitali theorem gives a countable pairwise disjoint subfamily covering the
remaining set up to a null set.  Since $\Psi_j^+\in L^1(\T)$, retain finitely many arcs so that the integral over the discarded tail is at most $\varepsilon_j$, and shrink them slightly to give positive mutual separation.  Absolute continuity of the integral allows the shrinking loss to be absorbed into $\varepsilon_j$.  Denote the resulting arcs by $J_{j,k}$ and their centers by $\zeta_{j,k}$.  Then
\begin{equation}\label{eq:sectorial-vitali}
 \left|
 \sum_k |J_{j,k}|\Psi_j(\zeta_{j,k})
 -\int_{B_j\setminus D_j}\Psi_j\,d\sigma
 \right|=o(1).
\end{equation}

Let $W_{j,k}$ denote the logarithmic weight of the target roots in the local
family chosen at $\zeta_{j,k}$.  The Step~1 data in the proof of
Lemma~\ref{lem:arc} have already been fixed above.  Execute Steps~2--5 of
that proof simultaneously for this finite family, taking its remaining
error tolerance equal to $\varepsilon_j$.  This gives
$H_j\in\cA(\Omega,p_j)$ such that
\[
 H_A(H_j)
 \le H_A(h_j)+\sum_k|J_{j,k}|W_{j,k}+o(1).
\]
Thus neither the source Cartan pair nor the harmonic cutoff is chosen until
after all target contours and the final arcs have been fixed.
For the absolute insertions, $W_{j,k}\le-\Psi_j(\zeta_{j,k})$.  The singular relative insertions have the additional loss $C_0$, whose total contribution is at most $C_0\sigma(B_j)=o(1)$ by \eqref{eq:sectorial-B}.  The regular relative insertions have loss $O(\log j)$ on a set of length $O(1/j)$, and hence also contribute $o(1)$.  Combining these bounds with \eqref{eq:sectorial-vitali} and the uninserted estimate \eqref{eq:sectorial-uninserted} yields
\begin{equation}\label{eq:sectorial-recovery}
 H_A(H_j)
 \le H_A(h_j)-\int_{B_j}\Psi_j\,d\sigma+o(1).
\end{equation}
The order of choices is important: at each fixed $j$, all target contours are fixed before the source Cartan pair and harmonic cutoffs in Lemma~\ref{lem:arc}, and the large root-distribution integer is chosen last.  Thus no Rouch\'e margin is required to be uniform in a parameter selected later.

\smallskip
\noindent\emph{Step 5: Jensen's formula is used at the moving center.}
Since $h_j(0)=p_j$ and $p_j\notin A$, Jensen's formula gives
\[
 \phi(p_j)+H_A(h_j)
 =\int_\T-\log|f(h_j)|\,d\sigma.
\]
On $\T\setminus B_j$ the integrand is at most
$\phi_j(g_j)+\eta_j$, while on $B_j$ one has
$-\log|f(h_j)|-\Psi_j=\phi_j(g_j)$.  Combining these identities with \eqref{eq:sectorial-recovery} and \eqref{eq:sectorial-nearmin} gives
\[
 \phi(p_j)+H_A(H_j)
 \le v_j(p)+o(1).
\]
The Riesz-envelope formula \eqref{eq:discG} therefore yields
\[
 \widetilde G_A(p_j)\le v_j(p)+o(1).
\]
Together with \eqref{eq:sectorial-center},
\[
 \widetilde G_A(p)\le v_j(p)+o(1).
\]
Since $v_j\le\widetilde G_A$, we obtain
$v_j(p)\to\widetilde G_A(p)$ for every $p\in\Omega\setminus A$; monotonicity gives the full increasing limit.

Put $w:=\sup_jv_j$.  Since $v_j\le\widetilde G_A$ and
$\widetilde G_A$ is locally bounded above, $w^*\in\PSH(\Omega)$ and
$w^*\le\widetilde G_A$.  The pointwise convergence just proved gives
$w=\widetilde G_A$ on $\Omega\setminus A$.  Both $w^*$ and
$\widetilde G_A$ are locally bounded above plurisubharmonic extensions of
the same function across the analytic hypersurface $A$; removability and
uniqueness of the extension therefore give
\[
 w^*=\widetilde G_A\quad\text{on }\Omega.
\]
This proves \eqref{eq:SID-limit} without any B-regularity assumption.

Assume now that $\Omega$ is B-regular.  Proposition~\ref{prop:PB-trunc}
gives $v_j=PB_{\phi_j}$, hence \eqref{eq:PB-limit}.  The increasing limit of
the continuous functions $v_j$ is lower semicontinuous on
$\Omega\setminus A$, while $\widetilde G_A$ is plurisubharmonic and therefore
upper semicontinuous; thus $\widetilde G_A$ is continuous there.  Finally,
\cite[Proposition~3.2]{LS99} gives local boundedness above near $A$, and so
\[
 G_{A,\Omega}=\log|f|+\widetilde G_A\longrightarrow-\infty
 \qquad\text{as }z\to A.
\]
This proves extended-real continuity.
\end{proof}

\section*{Open questions}

The preceding arguments leave several natural problems open.  We record the ones that seem closest to the methods developed here.

\begin{enumerate}
\item Let $\Omega\Subset\C^n$ be an arbitrary bounded B-regular domain and let $A=(f)|_\Omega$ be a reduced divisor extending across the boundary.  Is $G_{A,\Omega}$ necessarily continuous in the extended-real sense?  The results above answer this for convexifiable domains and, more generally, for pairs admitting an SID.  A proof on a general B-regular domain would have to replace the global holomorphic inward motion used in the approximation argument.

\item In particular, let $\Omega\Subset\C^n$, $n\ge2$, be a strongly
pseudoconvex domain with $C^2$ boundary, and let $f$ be holomorphic and
define a reduced divisor on a neighborhood of $\overline\Omega$.
For $A=(f)|_\Omega$, is $G_{A,\Omega}$ necessarily continuous in the
extended-real sense, without assuming convexifiability or an SID?
A stronger approximation question is whether
\[
 P_\Omega\min\{-\log|f|,j\}\nearrow\widetilde G_A
 \quad\text{on }\Omega\setminus A.
\]
The arguments of this paper do not settle these questions for an arbitrary
strongly pseudoconvex domain, even when its boundary is $C^\infty$.
The difficulty is to replace the global sectorial deformation while
retaining the quantitative control needed for divisor intersections near
the boundary.

\item What remains true when the pole set is an analytic set of codimension at least two, or more generally is defined by a coherent ideal sheaf?  The disc-functional definition of Green functions already exists in this setting, but the present insertion mechanism is intrinsically divisorial: it reduces the lost mass to zeros of one-variable Weierstrass polynomials.  A higher codimensional analogue would require a replacement for these scalar root weights which is stable under source-side gluing.

\item Which geometric conditions imply the SID property?  Example~\ref{ex:KN-sectorial} shows that convexifiability is not necessary.  It would be useful to identify broad classes of non-convexifiable domains carrying a global holomorphic flow which is inward near the boundary trace of the divisor and has finite-order escape at its singular boundary points.
\end{enumerate}

These questions also indicate a possible division of labor between the two
methods appearing in the paper.  The Hartogs-hull viewpoint packages the
boundary singularity into a geometric fiber function, while the disc method
recovers the missing logarithmic mass locally.  A setting in which these two
constructions can be compared beyond the hypotheses of the present theorems may provide a route beyond the present results.

\end{document}